\documentclass[12pt,a4paper]{amsart}
\usepackage[a4paper,inner=2.5cm,outer=2.5cm,top=2.5cm,bottom=2.5cm]{geometry}                
\usepackage{amsmath,amssymb,amscd,amsthm,amsfonts,anyfontsize,color,enumerate,fix-cm,layout,lipsum,lpic,mathrsfs,mdwlist,stmaryrd,tensor,tikz,thmtools,xspace,esint}
\usepackage{hyperref}
\usepackage[all]{xy}
\usepackage{mathtools}
\usepackage{xfrac}
\theoremstyle{plain}                 
\newtheorem{theorem}{Theorem}[section]     
\newtheorem{proposition}[theorem]{Proposition} 
\newtheorem{corollary}[theorem]{Corollary}     
\newtheorem{lemma}[theorem]{Lemma}        
\theoremstyle{definition}           
\newtheorem{definition}[theorem]{Definition}    
 
\theoremstyle{remark}       
\newtheorem{remark}[theorem]{Remark}    
\newtheorem{notation}{Notation}     
\hypersetup{colorlinks=true,linkcolor=blue,citecolor=purple,filecolor=magenta,urlcolor=cyan}

\newcommand{\Z}{\mathbb{Z}}

\newcommand{\C}{\mathbb{C}}
\newcommand{\wt}[1]{\widetilde{#1}}

\DeclareMathOperator{\Res}{Res}
\DeclareMathOperator{\Id}{Id}

\newcommand{\cord}[1]{\bigg{\langle} \, #1 \, \bigg{\rangle}}

\begin{document}

\title[Combinatorics of Bousquet-M\'elou--Schaeffer numbers]{Combinatorics of Bousquet-M\'elou--Schaeffer numbers in the light of topological recursion}

\author[B.~Bychkov]{B.~Bychkov}
\address{B.~B.: Faculty of Mathematics, National Research University Higher School of Economics, Usacheva 6, 119048 Moscow, Russia; and Demidov State University, 150003 Yaroslavl, Russia}
\email{bbychkov@hse.ru}

\author[P.~Dunin-Barkowski]{P.~Dunin-Barkowski}
\address{P.~D.-B.: Faculty of Mathematics, National Research University Higher School of Economics, Usacheva 6, 119048 Moscow, Russia; and ITEP, 117218 Moscow, Russia}
\email{ptdunin@hse.ru}

\author[S.~Shadrin]{S.~Shadrin}
\address{S.~S.: Korteweg-de Vries Institute for Mathematics, University of Amsterdam, Postbus 94248, 1090 GE Amsterdam, The Netherlands}
\email{S.Shadrin@uva.nl}

\begin{abstract} In this paper we prove, in a purely combinatorial way, a structural quasi-polynomiality property for the Bousquet-M\'elou--Schaeffer numbers. Conjecturally, this property should follow from the Chekhov-Eynard-Orantin topological recursion for these numbers (or, to be more precise, the Bouchard-Eynard version of the topological recursion for higher order critical points), which we derive in this paper from the recent result of Alexandrov-Chapuy-Eynard-Harnad. To this end, the missing ingredient is a generalization to the case of higher order critical points on the underlying spectral curve of the existing correspondence between the topological recursion and Givental’s theory for cohomological field theories.
\end{abstract}

\maketitle

\setcounter{tocdepth}{3}

\tableofcontents

\section{Introduction}

\subsection{Bousquet-M\'elou--Schaeffer numbers}
Bousquet-M\'elou--Schaeffer numbers~\cite{BS} (we call them the BMS numbers for brevity) can be considered as a special kind of Hurwitz numbers: they enumerate decompositions of a permutation of given cyclic type into a product of a given number of permutations of arbitrary cyclic types with a fixed total number of cycles in these permutations.

Throughout the paper we fix $m\geq 1$.
\begin{definition}
	
For $L\geq m$ and an integer partition $\mu$ we define
\begin{equation}
b^{\bullet,L}_{\mu} = \frac{|\mathrm{Aut}(\mu)|}{|\mu|!} \left|\left\{(\tau_1,\ldots,\tau_m) \, \Big| \, \tau_i \in S_{|\mu|};\, \tau_1\circ\ldots\circ\tau_m\in C_{\mu};  \sum\limits_{i=1}^m \ell(\tau_i) = L \right\}\right|.
\end{equation}
Here $C_\mu = C_\mu(S_{|\mu|})$ is the conjugacy class of the permutations with the cyclic type given by $\mu$, $\ell(\tau_i)$ denotes the number of cycles of the permutation $\tau_i\in S_{|\mu|}$, $i=1,\dots,m$, and $|\mathrm{Aut}(\mu)|$ denotes the order of the automorphism group of the set of parts of $\mu$.
\end{definition}

We define connected BMS numbers by the same formula, but with an additional requirement that the subgroup in $S_{|\mu|}$ generated by the tuple $\tau_1,\ldots,\tau_m$ acts on the set $\{1,\ldots,|\mu|\}$ transitively. Connected BMS numbers can also be defined as quantities of the isomorphism classes of ramified coverings of $\mathbb{C}P^1$ by Riemann surfaces of genus $g\geq 0$, with the ramification profile $\mu$ over $\infty\in \mathbb{C}P^1$ and ramification profiles $\tau_1,\dots,\tau_m$ over $m$ fixed finite points on $\mathbb{C}P^1$, which we chose to be $e^{2\pi\mathsf{i}\cdot j/m}$, $j=1,\dots,m$. In this case, the genus $g$ is related to the number $L$ and the parts $\mu_1,\ldots,\mu_n$, $n=\ell(\mu)$ of the partition $\mu$ via the Riemann--Hurwitz formula:
\begin{equation}
2g = 2 + (m-1)|\mu| - n - L.
\end{equation}
It is more natural for us to parametrize connected BMS number by genus $g$ rather than the number $L$:
\begin{notation}
We denote connected BMS numbers of genus $g$ by $b^\circ_{g,\mu}$.
\end{notation}
 
Bousquet-M\'elou and Schaeffer obtained in \cite{BS} a closed formula for the connected BMS numbers of genus $0$:
\begin{equation}\label{eq:BMSgenus0}
b^\circ_{0,\mu} = m \cdot \big((m-1)|\mu|-1\big)_{n-3}\cdot \prod\limits_{i=1}^{n}\binom{m\mu_i-1}{\mu_i};
\end{equation}
where we use the standard notation $(a)_b$ for the Pochhammer symbol, that is
\begin{equation} \label{eq:defPoch}
(a)_b \coloneqq \begin{cases}
 a(a-1)\cdots(a-b+1) & b> 0; \\
 1/(a+1)\cdots (a-b) & b < 0;\\
 1 & b=0.
\end{cases}
\end{equation}
In particular, for $n>3$ we have 
\begin{equation}
\big((m-1)|\mu|-1\big)_{n-3} =((m-1)|\mu|-1)((m-1)|\mu|-2)\ldots((m-1)|\mu|-n+3).
\end{equation}

\subsection{Combinatorial interpretations and connection to integrable systems}
The BMS numbers are interesting by themselves, and this type of enumeration problems is nowadays classical in combinatorics, see e.~g.~\cite{LandoZvonkin}, going back to the paper of Hurwitz~\cite{Hurwitz} (and the formula~\eqref{eq:BMSgenus0} is a generalization of a theorem of Hurwitz). But there is an extra motivation to study the BMS numbers that comes from a rich system of connections that they have with other areas of combinatorics and with integrable systems.  

Connections in combinatorics include an interpretation of the BMS numbers as enumeration of constellations, in the terminology of A.~Zvonkin. These constellations are basically the pictures that one can obtain through lifting the unit circle on $\mathbb{C}P^1$ considered as an $m$-gon with vertices $e^{2\pi\mathsf{i}\cdot j/m}$, $j=1,\dots,m$, labeled in the cyclic order, with the interior (resp., exterior) of the unit disk colored by black (resp., white) color. The enumeration of constellations was a part of the original motivation of Bousquet-M\'elou--Schaeffer, and it is a very active area of research, see e.~g.~\cite{Louf} for some recent results. There is also an interpretation of the BMS numbers in terms of the so-called Hurwitz numbers with $m$ strictly monotone blocks, which is a special case of more general Harnad-Orlov correspondence~\cite{HO} (see also \cite{ALS} for an exposition).

A connection with integrable systems is established via identification of a generating function of the BMS numbers with a particular tau-function of the KP hierarchy from the Orlov-Scherbin  (hypergeometric) family~\cite{OS,GJ,GuayHarnad,HO}, see also surveys in~\cite{KL,ALS,H}. Namely, let $\mathsf{cr}^\lambda = (\mathsf{cr}^\lambda_1,\dots,\mathsf{cr}^\lambda_{|\lambda|})$ denote the vector of contents of the standard Young tableau of a partition $\lambda$. The following KP tau-function belongs to the Orlov-Scherbin family
\begin{equation}
\label{eq:tau1}
Z\coloneqq\sum_{\lambda} \frac{\dim\lambda}{|\lambda|!} \prod_{i=1}^{|\lambda|} (1+\hbar \mathsf{cr}^\lambda_i)^m s_\lambda(p_1,p_2,\ldots)
\end{equation} 
(it is more convenient for us to use a rescaling of the standard KP variables $p_i = it_i$, $i=1,2,\dots,$) and gives an exponential generating function for the BMS numbers, that is,
\begin{equation}
\label{eq:tau2}
\log Z= \sum_{g=0}^\infty \sum_{\mu} \frac{\hbar^{2g-2+\ell(\mu)+|\mu|}}{|\mathrm{Aut}(\mu)|} b^\circ_{g,\mu} \prod_{i=1}^{\ell(\mu)} p_{\mu_i}.
\end{equation}








\subsection{Topological recursion} \label{sec:intro-toporec}

Topological recursion of Chekhov, Eynard, and Orantin~\cite{EynardOrantin,EynardOverview,EynardBook,LiuMulase} is a universal and a very powerful way to look at various enumerative problems in combinatorics and enumerative geometry, as, for instance, Hurwitz theory and Gromov-Witten theory, and it is the base for the remodelling of the B-model principle proposed in~\cite{BKMP}. 

The $n$-point generating differentials of the BMS numbers satisfy the topological recursion in the Bouchard-Eynard formulation~\cite{BouchardEynard}. It is a straightforward corollary of a more general result of Alexandrov, Chapuy, Eynard, and Harnad~\cite{ACEH}, and for $m=2$ it was proved by Kazarian and Zograf in~\cite{KZ2}. Let us explain this statement in detail. 

Let $z$ be a global coordinate on $\mathbb{C}P^1$ and consider the following function $x$ of $z$: 
\begin{equation}\label{eq:introx}
x:=\frac{(1+z)^m}z. 
\end{equation}
Our goal is to construct recursively a set of symmetric differentials $\omega_{g,n}(z_1,\dots,z_n)$, $g\geq 0$, $n\geq 1$, with the initial conditions 
\begin{equation}
\omega_{0,1}(z_1) \coloneqq -\frac{z_1^2}{(1+z_1)^m} dx(z_1) \qquad \text{and} \qquad
\omega_{0,2}(z_1,z_2)\coloneqq \frac{dz_1dz_2}{(z_1-z_2)^2}.
\end{equation}

Note that $x\colon \mathbb{C}P^1\to\mathbb{C}P^1$, $z\mapsto x(z)$, is a covering of degree $m$. Let $C$ denote the preimage of the positively oriented circle $|x|=R\gg 0$, where $R$ is chosen to be big enough so that the disk $|x|<R$ contains all critical values of $x$. The contour $C$ consists of two connected components that cover the circle $|x|=R$ with degrees $1$ and $m-1$ respectively, and it is oriented as the boundary of $x$-preimage of the disk $|x|<R$. For a point $\zeta\in C$ we denote by $\zeta_1, \zeta_2,\dots,\zeta_m$ the points in $x^{-1}(x(\zeta))$ with $\zeta_1=\zeta$.

The topological recursion reads:
\begin{align} \label{eq:topologicalrecursion}
& \omega_{g,n+1}(z_{[n]},z_{n+1}):= 
\\ \notag
& \frac{-1}{2\pi\mathsf{i}}\oint\limits_C \sum_{1\subsetneq I\subset [m]} \frac{\int\limits_o^{\zeta_1} \omega_{0,2}(\cdot, z_{n+1})}{\prod\limits_{i\in I} (\omega_{0,1}(\zeta_i)-\omega_{0,1}(\zeta_1))}
\sum_{\substack{
J\,\vdash\, I\cup \{1\} \\
\sqcup_{i=1}^{\ell(J)} N_i = [n] \\
\!\!\! \!
\sum_{i=1}^{\ell(J)} g_i = g +\ell(J) -|I|-1
}}^{\text{no}\, (0,1)}
\!\!\! \!
\prod_{i=1}^{\ell(J)} \omega_{g_i,|J_i|+|N_i|}(\zeta_{J_i},z_{N_i}),
\end{align}
where $o$ is an arbitrary base point, in the second sum we forbid the choices where we have $(g_i,|J_i|+|N_i|)=(0,1)$, by $[m]$ (resp., $[n]$) we denote the set $\{1,\dots,m\}$ (resp., $\{1,\dots,n\}$), and the notation of the form $\zeta_{J}$ means all variables $\zeta_i$, $i\in J$. Though it might not be obvious at the first glance, the right hand side of equation~\eqref{eq:topologicalrecursion} does not depend on the choice of the point $o$ and on the way we label $\zeta_2,\dots,\zeta_m$, and it is a symmetric function of $z_1,\dots,z_{n+1}$. 

The claim is that thus defined differentials $\omega_{g,n}$ are related to the BMS numbers via their expansion in the coordinate $X=1/x$ near $x=\infty$:
\begin{equation}\label{eq:expansionomega}
\omega_{g,n}(z_{[n]}) -\delta_{g,0}\delta_{n,2} \frac{dX_1dX_2}{(X_1-X_2)^2}  =\sum_{\mu_1,\dots,\mu_n=1}^\infty b^\circ_{g,\mu} \prod_{i=1}^n d X_i^{\mu_i},
\end{equation}
where $X_i=1/x(z_i)$, $i=1,\dots,n$. 

We prove that the BMS numbers can be reproduced via equations~\eqref{eq:introx}-\eqref{eq:expansionomega} (for a physicist it would mean that the spectral curve $x^2y=(1+xy)^m$, $y=\omega_{0,1}/dx$, provides the correct B-model for the BMS numbers); though, as we have already mentioned above, it is merely a small addendum to the argument of Alexandrov-Chapuy-Eynard-Harnad in~\cite{ACEH}, where they had to exclude the case of the BMS numbers mostly for the clarity of exposition. It is a very powerful statement that says a lot about the combinatorial structure of the BMS numbers (though the necessary piece of theory is not fully developed yet, see remarks~\ref{rem:BMS-ELSV}-\ref{rem:notexist} below), and it serves for us as a motivation and an inspiration to analyze the corresponding combinatorial structure of the BMS numbers independently in purely combinatorial terms.

\subsection{Combinatorial structure of the BMS numbers}

The fact that the BMS numbers satisfy the topological recursion suggests a very special combinatorial structure for them. Namely, it should impose that for $2g-2+n>0$, $n=\ell(\mu)$, the expression
\begin{equation} \label{eq:polynomialBMSexpression}
b^\circ_{g,\mu} \cdot \frac{\displaystyle\prod\limits_{i=1}^n\prod\limits_{\substack{m\leq j_i \leq 4g-4+2n-1 \\ m\nmid j_i}} \left(\mu_i-\dfrac{j_i}{m}\right) }{\displaystyle\prod\limits_{i=1}^n \dfrac{(m\mu_i-m)!}{\mu_i!\,(m\mu_i-\mu_i-1)!}}
\end{equation}
is given by a polynomial $\mathrm{Poly}_{g,n}(\mu_1,\dots,\mu_n)$. It is an absolutely remarkable property that can equivalently be formulated as a way to express the $n$-point generating functions for the BMS numbers given by $\sum_{\mu_1,\dots,\mu_n=1}^\infty b^\circ_{g,\mu} X_1^{\mu_1}\cdots X_n^{\mu_n}$ as finite linear combinations of the products of some explicitly given rational functions in $z_1,\dots,z_n$, where $X_i=z_i/(1+z_i)^m$, $i=1,\dots,n$. 

A natural question that immediately arises is whether one can prove such a remarkable property of the BMS numbers in a purely combinatorial way. Besides purely combinatorial motivation to do this, it is also a way to manifestly see the B-model side within the remodeling of the B-model philosophy as an emergent phenomenon, that is, the spectral curve defined by the function $x=x(z)$ (or rather the multi-valued function $z=z(x)$) emerges as a natural Riemann surface prescribed by the analytic behavior of the $n$-point generating functions of the underlying enumerative problem in all genera.

\emph{This is precisely the goal of this paper: our main result is the proof of Theorem \ref{th:apoly}, a purely combinatorial proof of the polynomiality of~\eqref{eq:polynomialBMSexpression} in $\mu_1,\dots,\mu_n$ (for every $g\geq 0$ such that $2g-2+n>0$). }

One can compare this with the analogous property of the usual connected Hurwitz numbers $h_{g,\mu}$. In that case, the polynomiality of the expressions
\begin{equation} \label{eq:HurwitzPolynomial}
h_{g,\mu} \cdot \prod_{i=1}^n \frac{\mu_i!}{\mu_i^{\mu_i}} 
\end{equation}
 for $2g-2+n>0$ was conjectured by Goulden-Jackson-Vainshtein in~\cite{GouldenJacksonVainshtein}. It was proved by Ekedahl-Lando-Shapiro-Vainshtein in~\cite{ELSV} using the ELSV formula that expresses Hurwitz numbers in terms of the intersection numbers on $\overline{\mathcal{M}}_{g,n}$. An equivalent reformulation is that $\sum_{\mu_1,\dots,\mu_n=1}^\infty h_{g,\mu} X_1^{\mu_1}\cdots X_n^{\mu_n}$ can be expressed as a finite linear combination of products of certain explicitly given rational functions in $z_1,\dots,z_n$, where $X_i=z_ie^{-z_i}$, $i=1,\dots,n$, was then derived as an intermediate step towards a more refined result by Goulden-Jackson-Vakil in~\cite{GouldenJacksonVakil}. 
 
 The following question, however, remained open: is there any purely combinatorial way to see the polynomiality of~\eqref{eq:HurwitzPolynomial} that would use the combinatorial definition of Hurwitz numbers in terms of the characters of the symmetric group? It took about 15 years before a purely combinatorial proof that wouldn't use the ELSV formula has appeared in~\cite{KLS} (see also~\cite{DBKOSS} for an earlier proof of the polynomiality of~\eqref{eq:HurwitzPolynomial} which is not entirely combinatorial but is also independent of the ELSV formula).
 
 \begin{remark}
 The present paper follows the line of papers \cite{KLS,KLPS,DPSS} (and, to a lesser extent, \cite{DBKOSS,DLPS}) where similar quasi-polynomiality properties were proved for other types of objects. Some general ideas are shared among these papers (including the present one), but the underlying combinatorics and the required proofs of the corresponding statements always turned out to be new and exciting.
 \end{remark}

\subsection{Organization of the paper} In Section~\ref{sec:toporec-proof} we show how to complete the argument of Alexandrov-Chapuy-Eynard-Harnad in order to prove the topological recursion for the BMS numbers, and explain what the topological recursion statement should imply for the combinatorics of the BMS numbers (modulo certain piece of theory that is missing in the literature as of now). 

In Section~\ref{sec:specfuncs} we study the so-called $\xi$-functions on the particular spectral curve corresponding to the BMS numbers. These $\xi$-functions are an important ingredient of the topological recursion theory and their structure provides the motivation for looking into the quasi-polynomiality structure of the BMS numbers.

In Section~\ref{sec:semiinf} we briefly recall the main facts from the semi-infinite wedge space (a.k.a. free-fermion) theory which are needed in the rest of the paper.

In Section~\ref{sec:aops} we introduce the $\mathcal{A}$-operators which allow to express the BMS numbers in terms of the semi-infinite wedge space correlators in a convenient way. We then study the properties of these $\mathcal{A}$-operators.

In Section~\ref{sec:bmsquasipol}, building upon the results of the preceding section, we prove the main theorem of the present paper which states the quasi-polynomiality of the BMS numbers.


Appendix \ref{sec:appfaul} is devoted to proving certain facts revolving around the classical Faulhaber's formula which are used in Section~\ref{sec:aops}.

\subsection{Acknowledgments}  B.~B. and P.~D.-B. thank Yu.~Burman, M.~Kazarian, and S.~Lando for useful discussions and also would like to acknowledge the warm hospitality of Korteweg-de Vries Institute for Mathematics. S.~S. thanks A.~Alexandrov, R.~Kramer, and D.~Lewanski for useful discussions. The research of B.~B. and P.~D.-B. 
was supported by the Russian Science Foundation (project 16-11-10316).

\section{Topological recursion for the BMS numbers} \label{sec:toporec-proof}
In this section we prove the main statement announced in Section~\ref{sec:intro-toporec} of the introduction. Namely, we prove that the differential forms defined by equations~\eqref{eq:introx}-\eqref{eq:topologicalrecursion} produce the BMS numbers via the expansion given by equation~\eqref{eq:expansionomega}. In order to make this statement a bit more precise, observe that $X(z)=1/x(z)=z/(1+z)^m$ can serve as a local coordinate in a neighborhood of $z=0$. 

\subsection{General statement}

\begin{proposition}\label{prop:ACEH} For the differential forms $\omega_{g,n}$ defined via equations~\eqref{eq:introx}-\eqref{eq:topologicalrecursion} we have:
\begin{equation}\label{eq:ACEHprop}
\Res\limits_{X_1=0} \cdots \Res\limits_{X_n=0} \frac{\omega_{g,n}(z_{[n]})-\delta_{g,0}\delta_{n,2} \frac{dX_1dX_2}{(X_1-X_2)^2}}{\prod_{i=1}^n \mu_i X_i^{\mu_i}} = b^\circ_{g,\mu},
\end{equation}
where $\mu=(\mu_1,\dots,\mu_n)$.
\end{proposition}

\begin{proof} Let us recall the main result of Alexandrov-Chapuy-Eynard-Harnad~\cite[Theorem 1.1]{ACEH} reformulated in an equivalent form
	in the spirit of~\cite{BouchardEynard} with the help of \cite[Theorem 5]{BouchardEynard}. 
	
Let $\epsilon = (\epsilon_1,\dots,\epsilon_m)\in \mathbb{C}^m$. It is convenient to assume that each $\epsilon_i$ belongs to a small open disk of radius $\rho\ll 1$ with the center at $1$, that is, $\epsilon\in U\subset \mathbb{C}^m$, where $U = \{ (\epsilon_1,\dots,\epsilon_m)\in \mathbb{C}^m\,|\, |\epsilon_i-1|<\rho \}$.

 Let $z$ be a global coordinate on $\mathbb{C}P^1$ and consider a function $x_\epsilon$ of $z$ that also depends on $\epsilon$ defined as 
\begin{equation}\label{eq:epsilon-x}
x(z;\epsilon) = \frac{\prod_{i=1}^m (1+\epsilon_i z)}{z}.
\end{equation}
Note that $x(\cdot;\epsilon)\colon \mathbb{C}P^1\to\mathbb{C}P^1$, $z\mapsto x(z;\epsilon)$, is a covering of degree $m$, and it is also a superpotential of a certain Frobenius manifold as discussed in detail in~\cite[Section 8]{DubrovinSuperpotential}. Let $C$ denote the preimage of the positively oriented circle $|x|=R\gg 0$, where $R$ is chosen to be big enough so that the disk $|x|<R$ contains all critical values of $x$. The contour $C$ consists of two connected components that cover the circle $|x|=R$ with degrees $1$ and $m-1$ respectively, and it is oriented as the boundary of $x(\cdot;\epsilon)$-preimage of the disk $|x|<R$. For a point $\zeta\in C$ we denote by $\zeta_1, \zeta_2,\dots,\zeta_m$ the points in $x(\cdot;\epsilon)^{-1}(x(\zeta;\epsilon))$ with $\zeta_1=\zeta$.

We construct recursively a set of symmetric differentials $\omega_{g,n}(z_1,\dots,z_n;\epsilon)$, $g\geq 0$, $n\geq 1$, with the initial conditions 
\begin{equation}
\omega_{0,1}(z_1;\epsilon) \coloneqq -\frac{z_1^2}{\prod_{i=1}^m (1+\epsilon_i z)} dx(z_1) \qquad \text{and} \qquad
\omega_{0,2}(z_1,z_2;\epsilon)\coloneqq \frac{dz_1dz_2}{(z_1-z_2)^2},
\end{equation}
using the following recursion:
\begin{align} \label{eq:epsilon-topologicalrecursion}
& \omega_{g,n+1}(z_{[n]},z_{n+1};\epsilon):= 
\\ \notag
& \frac{-1}{2\pi\mathsf{i}}\oint\limits_C \sum_{1\subsetneq I\subset [m]} \frac{\int\limits_o^{\zeta_1} \omega_{0,2}(\cdot, z_{n+1})}{\prod\limits_{i\in I} (\omega_{0,1}(\zeta_i)-\omega_{0,1}(\zeta_1))}
\sum_{\substack{
		J\,\vdash\, I\cup \{1\} \\
		\sqcup_{i=1}^{\ell(J)} N_i = [n] \\
		\!\!\! \!
		\sum_{i=1}^{\ell(J)} g_i = g +\ell(J) -|I|-1
}}^{\text{no}\, (0,1)}
\!\!\! \!
\prod_{i=1}^{\ell(J)} \omega_{g_i,|J_i|+|N_i|}(\zeta_{J_i},z_{N_i};\epsilon),
\end{align}

Then \cite[Theorem 1.1]{ACEH} implies that \emph{outside the discriminant of $x(\cdot,\epsilon)$ in $U$} (they had to exclude the discriminant for technical reasons)
\begin{equation}\label{eq:epsilon-residues}
\Res\limits_{X_1=0} \cdots \Res\limits_{X_n=0} \frac{\omega_{g,n}(z_{[n]};\epsilon)-\delta_{g,0}\delta_{n,2} \frac{dX_1dX_2}{(X_1-X_2)^2}}{\prod_{i=1}^n \mu_i X_i^{\mu_i}} = b^\circ_{g,\mu}(\epsilon),
\end{equation}
where $\mu=(\mu_1,\dots,\mu_n)$ and $b^\circ_{g,\mu}(\epsilon)$ are defined through the expansion of the logarithm of the following Orlov-Scherbin tau-function
\begin{equation}
Z(\epsilon)\coloneqq\sum_{\lambda} \frac{\dim\lambda}{|\lambda|!} \prod_{i=1}^{|\lambda|} \prod_{j=1}^m (1+\hbar \epsilon_j \mathsf{cr}^\lambda_i) s_\lambda(p_1,p_2,\ldots)
\end{equation} 
as
\begin{equation}
\log Z(\epsilon)= \sum_{g=0}^\infty \sum_{\mu} \frac{\hbar^{2g-2+\ell(\mu)+|\mu|}}{|\mathrm{Aut}(\mu)|} b^\circ_{g,\mu}(\epsilon) \prod_{i=1}^{\ell(\mu)} p_{\mu_i}.
\end{equation}

Note that the point $\epsilon=(1,\dots,1)$ does belong to the discriminant, so the statement of \cite[Theorem 1.1]{ACEH} doesn't work directly in this case. However, we intentionally reformulated that theorem in a way where it is absolutely manifest from the construction and equations~\eqref{eq:epsilon-x}-\eqref{eq:epsilon-topologicalrecursion} that both the left hand side and the right hand side of equation~\eqref{eq:epsilon-residues} depend analytically on $\epsilon\in U$. Thus statement \eqref{eq:epsilon-residues} can be extended to the discriminant, and the analyticity implies the continuity of the values, that is, 
\begin{equation}
\Res\limits_{X_1=0} \cdots \Res\limits_{X_n=0} \frac{\omega_{g,n}(z_{[n]};(1,\dots,1))-\delta_{g,0}\delta_{n,2} \frac{dX_1dX_2}{(X_1-X_2)^2}}{\prod_{i=1}^n \mu_i X_i^{\mu_i}} = b^\circ_{g,\mu}(1,\dots,1) = b^\circ_{g,\mu},
\end{equation}
which proves equation~\eqref{eq:ACEHprop}.
\end{proof}

\begin{remark}\label{rem:BMS-ELSV}
Proposition~\ref{prop:ACEH} and a proper generalization of the results of~\cite{EynardIntersection,DOSS} (see also the expositions in~\cite{DubrovinSuperpotential,LPSZ,Lewanski}) would imply that the BMS numbers can be expressed in terms of the intersection numbers of a cohomological field theory with a non-flat unit. This cohomological field theory is not particularly nice, though it is remotely related to the one discussed in ~\cite[Section 8]{DubrovinSuperpotential}. It can be represented as a particular element of the Givental group action applied to the direct sum of the rescaled Witten $m$-spin class and a trivial TFT, with deformed dilaton leaves. The corresponding ELSV-type formula for $b^\circ_{g,\mu}$ would look like a polynomial in $\mu_1,\dots,\mu_n$ with an explicit non-polynomial factor.
\end{remark}

\begin{remark}\label{rem:remark2xi} Equivalently, Proposition~\ref{prop:ACEH} and a proper generalization of the results of~\cite{EynardIntersection,DOSS} (see also the necessary local analysis near a branching point of higher order in~\cite[Section 7]{DubrovinSuperpotential}, \cite{MilanovLew}, and~\cite{BoucharEynard-2}) would imply that for $2g-2+n>0$ the formal $n$-point functions $\sum_{\mu_1,\dots,\mu_n=1}^\infty h_{g,\mu} X_1^{\mu_1}\cdots X_n^{\mu_n}$ are the Taylor series expansions near $X_1=\cdots=X_n=0$ of a finite linear combination of the products of the derivatives of certain globally defined $\xi$-functions that we introduce in the next section. 
\end{remark}

\begin{remark} \label{rem:notexist}
The necessary piece of theory to deal with ELSV-type formulas and topological recursion / cohomological field theory correspondence at the discriminant (which we need to make the statements of the Remarks~\ref{rem:BMS-ELSV} and~\ref{rem:remark2xi} more explicit and to prove them) doesn't exist in the literature at the moment. However, we expect that it should be rather straightforward to prove the appropriate generalizations of the statements of~\cite{EynardIntersection,DOSS}, and some important papers in this direction include~\cite{Janda,BorotBou}.
\end{remark}

These remarks indicate a missing piece of theory that, once completed, would imply severe restrictions on the possible dependence of $b^\circ_{g,\mu}$ on the parameters $\mu_1,\dots,\mu_n$, $n=\ell(\mu)$. Namely, these remarks imply the polynomiality of \eqref{eq:polynomialBMSexpression}. In the rest of the paper we prove this polynomiality independently, via purely combinatorial methods, and explain explicitly why it is equivalent to the statement of Remark~\ref{rem:remark2xi} (see Corollary~\ref{corollaryW}).

\subsection{Unstable cases} Let us also, for completeness, derive equation~\eqref{eq:expansionomega} for the unstable cases $(g,n) \in \{(0,1),(0,2)\}$ directly. 

From the explicit formula for the genus $0$ BMS-numbers from \cite{BS} (formula \eqref{eq:BMSgenus0}), or equivalently from formula \eqref{eq:bAexpr} we get:
\begin{align}
&b^\circ_{0,k_1} = \frac{m(mk_1-1)!}{k_1!(mk_1-k_1+1)!}\\
&b^\circ_{0,k_1k_2} = \frac{m}{m(k_1+k_2)-k_1-k_2}\binom{mk_1-1}{k_1}\binom{mk_2-1}{k_2} \label{eq:bms02}
\end{align}

\begin{proposition}
	The expansion of $\omega_{0,1}$ is given by:
	\begin{equation}
	-\frac{z_1^2}{(1+z_1)^m} dx_1=d\sum_{k=1}^\infty b^\circ_{0,\mu_1} X_1^{\mu_1}
	\end{equation}
\end{proposition}
\begin{proof}
	
	We have: 
	\begin{equation}
	 -\frac{z^2}{(1+z)^m} dx = \dfrac{z}{X}dX
	\end{equation}
Note that	
	\begin{align}
	[X^k]z &= \oint \frac{z}{X^{k+1}} dX =\oint \dfrac{(1+z)^{mk+m}}{z^k}\dfrac{1+z-mz}{(1+z)^{m+1}}dz =\oint \dfrac{(1+z)^{mk-1} (1+z-mz)}{z^k}dz \\ \nonumber 
	&= \binom{mk}{k-1} - m \binom{mk-1}{k-2} = \dfrac{m(mk-1)!}{(k-1)!(mk-k+1)!}.
	\end{align}
Therefore, the $X$ expansion of $\omega_{0,1}$ is given by:
	\begin{equation}
	\omega_{0,1} = \sum\limits_{k=1}^{\infty} k b^\circ_{0,k} X^{k-1}dX = d \sum\limits_{k=1}^{\infty} b^\circ_{0,k} X^k.
	\end{equation}
\end{proof}

\begin{proposition} 	The expansion of $\omega_{0,2}$ with the subtracted singularity at the diagonal is given by:
	\begin{equation}
	\frac{dz_1dz_2}{(z_1-z_2)^2} - \frac{dX_1dX_2}{(X_1-X_2)^2}= d_1 d_2 \sum_{\mu_1,\mu_2=1}^\infty b^\circ_{0,\mu_1,\mu_2} X_1^{\mu_1} X_2^{\mu_2}. 
	\end{equation}
\end{proposition}
\begin{proof}
	Let us denote $z_i = z(X_i), i=1,2$. It is sufficient to prove that
	\begin{equation}
	\label{eq:log}
	\mathrm{log}(z_1-z_2) = \mathrm{log}(X_1-X_2)+C(X_1)+C(X_2) + \sum\limits_{k_1,k_2\geq 0} b^\circ_{0,k_1^1k_2^1} X^{k_1}X^{k_2}.
	\end{equation}
	Let us apply the Euler operator $E := X_1\dfrac{\partial}{\partial X_1} + X_2\dfrac{\partial}{\partial X_2}$ to the both sides of \eqref{eq:log}.
	Recall that
	\begin{align}
	X(z) = \frac{z}{(1+z)^m} \qquad \text{and} \qquad
	\frac{d}{dX} = \frac{(1+z)^{m+1}}{1+z-mz}\frac{d}{dz}.
	\end{align}
We have:	
	\begin{align}
	E\,\mathrm{log}(z_1-z_2) & = \frac{z_1}{(1+z_1)^m}\frac{(1+z_1)^{m+1}}{1+z_1-mz_1}\frac{1}{z_1-z_2} - \frac{z_2}{(1+z_2)^m}\frac{(1+z_2)^{m+1}}{1+z_2-mz_2}\frac{1}{z_1-z_2}\\
	\notag &=
	\frac{(z_1+z_1^2)(1+z_2-mz_2) - (z_2+z_2^2)(1+z_1-mz_1)}{(z_1-z_2)(1+z_1-mz_1)(1+z_2-mz_2)} \\
	\notag & =
	\frac{(1+z_1)(1+z_2)-mz_1z_2}{(1+z_1-mz_1)(1+z_2-mz_2)}.
	\end{align}
	Now we compute the coefficient $[X_1^{k_1}X_2^{k_2}]E\,\mathrm{log}(z_1-z_2)$ as a residue:
	\begin{align}
	\label{eq:Elogz}
	& [X_1^{k_1}X_2^{k_2}]E\,\mathrm{log}(z_1-z_2) \\ \notag & 
	=  \oiint  \frac{(1+z_1)(1+z_2)-mz_1z_2}{(1+z_1-mz_1)(1+z_2-mz_2)}\frac{(1+z_1)^{m(k_1+1)}(1+z_2)^{m(k_2+1)}}{z_1^{k_1+1}z_2^{k_2+1}}
	\\ \notag 
	& \phantom{=  \oiint  {}} \times \frac{(1+z_1-mz_1)(1+z_2-mz_2)}{(1+z_1)^{m+1}(1+z_2)^{m+1}} dz_1dz_2
	\\ \notag &
	=\oiint \frac{((1+z_1)(1+z_2)-mz_1z_2)(1+z_1)^{mk_1-1}(1+z_2)^{mk_2-1}}{z_1^{k_1+1}z_2^{k_2+1}}dz_1dz_2
	\\ \notag &
	=\binom{mk_1}{k_1}\binom{mk_2}{k_2} - m\binom{mk_1-1}{k_1-1}\binom{mk_2-1}{k_2-1}
	\end{align}
We also have
	\begin{equation}
	\label{eq:Ebms}
	[X_1^{k_1}X_2^{k_2}]E\,\sum\limits_{k_1,k_2\geq 0} b^\circ_{0,k_1^1k_2^1} X^{k_1}X^{k_2} = (k_1+k_2)b^\circ_{0,k_1^1k_2^1}
	\end{equation}
	From \eqref{eq:Elogz} and \eqref{eq:Ebms}, taking into account \eqref{eq:bms02}, we immediately see that equation~\eqref{eq:log} holds.
\end{proof}

\section{Functions on the spectral curve}\label{sec:specfuncs}

Recall that the spectral curve is defined as the source curve of the covering $x\colon \mathbb{C}P^1\to \mathbb{C}P^1$, $z\mapsto x(z)$, where
\begin{align}
\label{eq:yz}
x&=\dfrac{(1+z)^m}{z},
\end{align}
and we are interested in the expansions of functions on the curve in the variable
\begin{equation}
\label{eq:Xz}
X=\dfrac{1}{x}=\dfrac{z}{(1+z)^m}.
\end{equation}
We define the $\xi$-functions on the curve by explicit formulas:
\begin{equation}
\xi_i := \dfrac{z^i}{(1+z)^{m-1}\,(-1+(m-1)z)}, \qquad i=0\dots m-1.
\label{Def:xi}
\end{equation}

Let $D$ denote the operator $\frac{d}{dx}$. Consider the space of functions on the curve spanned by $\xi$-functions and their $x$-derivatives up to the order $d\geq 0$:
\begin{equation}
\Xi^d := \left\langle \xi_0,\dots\xi_{m-1},D\xi_0,\dots,D\xi_{m-1},\dots,D^d\xi_0,\dots,D^d\xi_{m-1}\right\rangle.
\end{equation}
We want to characterize the possible series expansions of functions from $\Xi^d$ in the variable $X$ at $X\rightarrow 0$. To this end, we have the following proposition.

\begin{proposition}\label{prop:xiexp}
	Let $P(k)$ be an arbitrary polynomial in $k$ of degree~$\leq m(d+1)-1$. Then there exists a unique function $\xi \in \Xi^d$ such that for all $k\in\mathbb{Z}_{\geq 0}$
	\begin{equation} \label{eq:xiexp}
	[X^k]\xi=\dfrac{(mk-m)!}{k!\,(mk-k-1)!} \; \dfrac{P(k)}{\displaystyle\prod_{\substack{m< j < m(d+1) \\ m\nmid j}} \left(k-\dfrac{j}{m}\right)}.
	\end{equation}
	Here $[X^k]\xi$ denotes the coefficient of $X^k$ in the series expansion of $\xi$ in the variable $X$ at $X\to 0$.
\end{proposition}
\begin{proof}
First of all, let us find the series expansion of $\xi_i$. Note that
\begin{equation}
\xi_i=\dfrac{z^{i-2}}{dx/dz}=\dfrac{z^{-2}}{i+1}\dfrac{d}{dx}z^{i+1}\qquad i=0\dots m-1.
\end{equation}
We have:
\begin{align} \label{eq:xizeroexp}
[X^k]\xi_i &= \oint X^{-k-1} \xi_i dX = -\oint x^{k-1} \dfrac{z^{-2}}{i+1}\dfrac{d}{dx}z^{i+1} dx = -\oint z^{i-2} \dfrac{(1+z)^{m(k-1)}}{z^{k-1}} d z \\ \nonumber
& =\binom{mk-m}{k-i}.
\end{align}

We prove the existence part of the proposition by induction (we will address the uniqueness part separately). From \eqref{eq:xizeroexp} it is easy to see that the base of induction holds, i.~e. that the existence  statement is true for $d=0$.

Let us prove the induction step. Apply operator $D=-X^2\dfrac{d}{d X}$ 
to a series $\xi$ of the form \eqref{eq:xiexp}, where $P(k)$ is an arbitrary polynomial of degree~$\deg(P)\leq m(d+1)-1$. We have:
\begin{align}\label{eq:Dxiexp}
& [X^k]D\xi=-(k-1)\dfrac{(mk-2m)!}{(k-1)!\,(mk-m-k)!} \; \dfrac{P(k-1)}{\displaystyle\prod_{\substack{m< j < m(d+1) \\ m\nmid j}} \left(k-1-\dfrac{j}{m}\right)}\\ \nonumber
&=-\dfrac{(mk-m)!}{k!\,(mk-k-1)!} \; \dfrac{P(k-1)\; k\; (mk-k-1)(mk-k-2)\cdots(mk-k-m+1)}{\displaystyle m^m\; \prod_{m < i < 2m } \left(k-\dfrac{i}{m}\right) \prod_{\substack{m< j < m(d+1) \\ m\nmid j}} \left(k-\dfrac{m+j}{m}\right)}\\ \nonumber
&=-\dfrac{(mk-m)!}{k!\,(mk-k-1)!} \; \dfrac{P(k-1)\; k\;(m-1)^{m-1} \displaystyle\prod_{i=1}^{m-1} \left(k-\dfrac{i}{m-1} \right)}{m^m\; \displaystyle\prod_{\substack{m< j < m(d+2) \\ m\nmid j}} \left(k-\dfrac{j}{m}\right)}.
\end{align}

Recall that $\Xi^{d+1}=\Xi^d + D\Xi^d$. Consider two functions $\xi_1,\xi_2\in\Xi^d$. From the induction hypothesis we have:
\begin{align}
[X^k]\xi_1&=\dfrac{(mk-m)!}{k!\,(mk-k-1)!} \; \dfrac{P_1(k)}{\displaystyle\prod_{\substack{m< j < m(d+1) \\ m\nmid j}} \left(k-\dfrac{j}{m}\right)},\\
[X^k]\xi_2&=\dfrac{(mk-m)!}{k!\,(mk-k-1)!} \; \dfrac{P_2(k)}{\displaystyle\prod_{\substack{m< j < m(d+1) \\ m\nmid j}} \left(k-\dfrac{j}{m}\right)},
\end{align}
where $P_1(k)$ and $P_2(k)$ can be any polynomials in $k$ of degree~$\leq m(d+1)-1$. Equation~\eqref{eq:Dxiexp} implies that 
\begin{equation}
[X^k](\xi_1+D\xi_2)=\dfrac{(mk-m)!}{k!\,(mk-k-1)!} \; \dfrac{Q(k)}{\displaystyle\prod_{\substack{m< j < m(d+2) \\ m\nmid j}} \left(k-\dfrac{j}{m}\right)},
\end{equation}
where
\begin{align}\label{eq:QP1P2}
Q(k) =\ & P_1(k)\; \cdot \; \prod_{i=(d+1)m+1}^{(d+2)m-1} \left(k-\dfrac{i}{m}\right) \\ \notag & + P_2(k-1)\; \cdot \;  (-m^{-m} (m-1)^{m-1})\prod_{i=0}^{m-1} \left(k-\dfrac{i}{m-1}\right) .
\end{align}
Since $\prod_{i=(d+1)m+1}^{(d+2)m-1} \left(k-\frac{i}{m}\right)$ and $\prod_{i=0}^{m-1} \left(k-\frac{i}{m-1}\right)$ are coprime polynomials of degrees~$m-1$ and~$m$ respectively, we can choose polynomials $P_1$ and $P_2$ of degree~$\leq m(d+1)-1$ to represent via equation~\eqref{eq:QP1P2} any polynomial $Q$ of degree~$\leq m(d+2)-1$. This proves the induction step and thus the existence part of the proposition.

In order to prove the uniqueness claim of the proposition, it is sufficient to observe that $\dim \Xi^d = m(d+1)$ is equal to  the dimension of the space of polynomials in $k$ of degree~$\leq m(d+1)-1$. This completes the proof of the proposition.
\end{proof}

\section{Semi-infinite wedge formalism} \label{sec:semiinf}

In this section we briefly recall the notion of semi-infinite wedge formalism; it is nowadays a standard tool in Hurwitz theory. For a more complete introduction see e.g.~\cite{Joh}.

\subsection{Basic definitions}

\begin{definition}
The Lie algebra $A_\infty$ is the $\C$-vector space of matrices $(A_{i,j})_{i,j\in \Z+\frac{1}{2}}$ with
only finitely many non-zero diagonals (that is, $A_{i,j}$ is not equal to zero only for finitely
many possible values of $i-j$), together with the commutator bracket. In this algebra we consider the following elements:
\begin{enumerate}
	\item The standard basis is the set $\{E_{i,j}\,|\, i,j \in \Z+\frac{1}{2}\}$ such that $(E_{i,j})_{k,l} = \delta_{i,k}\delta_{j,l}$;
	\item The diagonal elements $\mathcal{F}_n = \sum_{k\in\Z+\frac{1}{2}} k^n E_{k,k}$. In particular the element $C\coloneqq\mathcal{F}_0$ is called the {\it charge operator} and the element $E \coloneqq \mathcal{F}_1$ is called the {\it energy operator} (an element $A\in A_\infty$ has energy $e\in\Z$ if $[A,E]~=~eA$);
	\item For any non-zero integer $n$, the energy $n$ element $\alpha_n = \sum\limits_{k\in\Z+\frac{1}{2}} E_{k-n,k}.$
\end{enumerate}
\end{definition}

We 
construct a certain projective representation of this algebra, called the semi-infinite wedge space.

Let $V$ be an infinite-dimensional complex vector space with a basis labeled by half-integers. Denote the basis vector labeled by $m/2$ by $\underline{m/2}$, so $V = \bigoplus_{i \in \Z + \frac{1}{2}} \C \underline{i}$.

\begin{definition}\label{DEFSEMIINF}
The semi-infinite wedge space $\bigwedge^{\frac{\infty}{2}}(V) = \mathcal{V}$ is defined to be the span of all of the semi-infinite wedge products of the form
\begin{equation}\label{wedgeProduct}
\underline{i_1} \wedge \underline{i_2} \wedge \cdots
\end{equation}
for any decreasing sequence of half-integers $i_k$ such that there is an integer $c$ with $i_k + k - \frac{1}{2} = c$ for $k$ sufficiently large. The constant $c$ is called the \textit{charge}. We give $\mathcal{V}$ an inner product $(\cdot,\cdot)$ declaring its basis elements to be orthonormal.
\end{definition}

\begin{remark}
By definition \ref{DEFSEMIINF} the charge-zero subspace $\mathcal{V}_0$ of $\mathcal{V}$ is spanned by semi-infinite wedge products of the form 
\begin{equation}
v_\lambda = \underline{\lambda_1 - \frac{1}{2}} \wedge \underline{\lambda_2 - \frac{3}{2}} \wedge \cdots
\end{equation}
for some integer partition~$\lambda$. Hence we can identify integer partitions with the basis of this space:
\begin{equation}
\mathcal{V}_0 = \bigoplus_{n \in \mathbb{N} } \bigoplus_{\lambda\,  \vdash\, n} \C v_{\lambda}
\end{equation}
\end{remark}

The empty partition $\emptyset$ plays a special role.
We call 
\begin{equation}
v_{\emptyset} = \underline{-\frac{1}{2}} \wedge \underline{-\frac{3}{2}} \wedge \cdots
\end{equation}
 the vacuum vector and we denote it by $|0\rangle$. Similarly we call the covacuum vector its dual with respect to the scalar product $(\cdot,\cdot)$ and we denote it by $\langle0|$.

\begin{definition}
 The \emph{vacuum expectation value} or \emph{disconnected correlator} $\langle \mathcal{P}\rangle$ of an operator~$\mathcal{P}$ acting on $\mathcal{V}_0$ is defined to be:
\begin{equation}
\langle \mathcal{P}\rangle \coloneqq (|0\rangle , \mathcal{P} |0\rangle) \eqqcolon \langle 0 |\mathcal{P}|0 \rangle
\end{equation}
\end{definition}

The \emph{connected correlators} $\langle \mathcal{P}\rangle^\circ$ are defined through the disconnected ones in the usual way with the help of the inclusion-exclusion formula.

\begin{definition}
Define a projective representation of $A_\infty$ on $\mathcal{V}_0$ as follows: for $i\neq j$ or
$i = j > 0$, $E_{i,j}$ checks whether $v_\lambda$ contains $\underline{j}$ as a factor and replaces it by $\underline{i}$ if it does. If $i = j < 0$, $E_{i,j}v_\lambda = -v_\lambda$ if $v_\lambda$ does not contain $\underline{j}$. In all other cases it gives zero.
\end{definition}

Equivalently, this gives a representation of the central extension $A_{\infty}+\C\mathrm{Id}$,
with commutation relations between the basis elements given by 
\begin{equation}
\label{E:comm}
[E_{a,b},E_{c,d}] = \delta_{b,c}E_{a,d} - \delta_{a,d}E_{c,b}+\delta_{b,c}\delta_{a,d}(\delta_{b>0}-\delta_{d>0})\mathrm{Id}.
\end{equation}

The operator $E_{i,j}$ has energy $j-i$, hence all the $\mathcal{F}_n$'s have zero energy. Operators with positive energy annihilate the vacuum while negative energy operators are annihilated by the covacuum. 

Note that $[\alpha_k , \alpha_l] = k\delta_{k+l,0}$.


Also let us denote by $(\Delta\,f)(l) = f(l)-f(l-1)$ the difference operator acting on functions of $l$.
Using the commutation rule (\ref{E:comm}), it is useful to obtain the following: 
\begin{lemma}
\label{l:Ecomm}
For any integer $a$
\begin{equation}
\left[\sum\limits_{l\in\Z+\frac{1}{2}} E_{l-1,l},\sum\limits_{l\in\Z+\frac{1}{2}} f(l)E_{l+a,l} \right] =
\sum\limits_{l\in\Z+\frac{1}{2}} (\Delta\,f)(l) E_{l+a-1,l} + \delta_{a,1}f_{-1/2}\mathrm{Id}.
\end{equation}
\end{lemma}

\subsection{The BMS numbers}
As we mentioned in the introduction (see \eqref{eq:tau1}, \eqref{eq:tau2}) the BMS numbers have the following expression in terms of fermionic Fock space operators:

 \begin{equation}
\sum\limits_{g=0}^\infty\sum\limits_{\mu} \frac{\hbar^{2g-2+l(\mu)+|\mu|}}{|\mathrm{Aut}(\mu)|}b^\circ_{g,\mu}\prod\limits_{i=1}^{l(\mu)}p_{\mu_i} 
= \sum\limits_{|\mu|}\sum\limits_{\mu\vdash |\mu|}\prod\limits_{i=1}^{|\mu|} (1+\hbar\mathrm{cr}_i^\mu)^m \frac{\dim\mu}{|\mu|!}s_\mu(p_1,p_2,\ldots)
\end{equation} 

Let us denote by $D(\hbar)$ an operator acting on the basis of the charge zero sector of the Fock space with the eigenvalues equal to the generating series for the elementary symmetric polynomials $\sigma_k(x_1,\ldots,x_n)$:
\begin{equation}
D(\hbar)v_\lambda = \sum\limits_{k=0}^\infty \sigma_k(c(\lambda))\hbar^k v_\lambda = \prod\limits_{i,j\in\lambda} (1+\hbar(j-i))v_\lambda.
\end{equation}
Here $c(\lambda)$ is a vector of contents of the diagram $\lambda$ and in the last identity we used the standard expression for the generating series of the elementary symmetric polynomials:
\begin{equation}
\sum\limits_{k=0}^{\infty} \sigma_k(x_1,\ldots,x_n) t^k = \prod\limits_{i=1}^n (1+tx_i).
\end{equation}
Now we can rewrite the BMS numbers  as the following vacuum expectation value:
  
\begin{align}
\sum\limits_{g=0}^\infty\sum\limits_{\mu} \frac{\hbar^{2g-2+l(\mu)+|\mu|}}{|\mathrm{Aut}(\mu)|}b^\circ_{g,\mu}\prod\limits_{i=1}^{l(\mu)}p_{\mu_i}= \left\langle e^{\alpha_1}D(\hbar)^m e^{\sum\limits_{i=1}^\infty\frac{\alpha_{-i}p_{i}}{i}} \right\rangle^{\circ};\\
 b^\circ_{g,\mu} = [\hbar^{2g-2+l(\mu)+|\mu|}] \left\langle e^{\alpha_1}D(\hbar)^m\prod\limits_{i=1}^n\frac{\alpha_{-\mu_i}}{\mu_i} \right\rangle^{\circ}.
\end{align}
The operators $e^{\alpha_1}$ and $D(\hbar)^m$ act identically on the vacuum. Therefore, we rewrite the last formula as
\begin{equation}\label{eq:boperexpr}
 b^\circ_{g,\mu} = [\hbar^{2g-2+l(\mu)+|\mu|}] \left\langle e^{\alpha_1}D(\hbar)^m\prod\limits_{i=1}^n\frac{\alpha_{-\mu_i}}{\mu_i}D(\hbar)^{-m}e^{-\alpha_1} \right\rangle^{\circ}.
\end{equation}

\section{The $\mathcal{A}$-operators}\label{sec:aops}

Let us define the following operators:
\begin{definition}\label{def:acheck}
	For a positive integer $k$
	\begin{equation} \label{eq:aopdef}
	\check{\mathcal{A}}(k,\hbar):= \hbar^{-k}e^{\alpha_1}D(\hbar)^m\dfrac{\alpha_{-k}}{k}D(\hbar)^{-m}e^{-\alpha_1}
	\end{equation}
\end{definition}
With the help of this definition we can reformulate \eqref{eq:boperexpr} as follows:
\begin{proposition}	We have:
	\begin{equation}\label{eq:bAexpr}
	b^\circ_{g,\mu} = [\hbar^{2g-2+l(\mu)}] \left\langle\prod\limits_{i=1}^n\check{\mathcal{A}}(\mu_i,\hbar) \right\rangle^{\circ}.
	\end{equation}
\end{proposition} 

\subsection{Quasi-rationality of the $\check{\mathcal{A}}$-operators}
In this subsection we prove the quasi-rationality of the coefficients of the $\check{\mathcal{A}}$-operators.

\subsubsection{The $\check{\mathcal{A}}$-operator via the difference operator}
We need to prove the following technical lemmata first:
\begin{lemma}\label{L:conj1} We have:
\begin{equation}
D^m(\hbar)\alpha_{-k}D^{-m}(\hbar) =\sum\limits_{l\in\Z+\frac{1}{2}} \prod_{i=1}^k (1+\hbar(l+i-1/2))^mE_{l+k,l}.
\end{equation}
\end{lemma}
\begin{proof}
Let $D = \sum\limits_{l\in\Z+\frac{1}{2}}d_lE_{l,l}$, where $\dfrac{d_{l+1}}{d_l} = 1+\hbar(l+1/2)$.
Let us track the element $\underline{l}$ under the action of $D^m(\hbar)\alpha_{-k}D^{-m}(\hbar)$:
\begin{equation}
\underline{l} \stackrel{D^{-1}}{\longmapsto} d^{-1}_l\underline{l} \stackrel{\alpha_{-k}}{\longmapsto} d^{-1}_l\underline{l+k} \stackrel{D}{\longmapsto} \frac{d_{l+k}}{d_l}\underline{l+k}.
\end{equation} 
It remains to be noticed that 
$\frac{d_{l+k}}{d_l} = \prod\limits_{i=1}^{k}\frac{d_{l+i}}{d_{l+i-1}}= \prod\limits_{i=1}^k (1+\hbar(l+i-1/2)).$
\end{proof}

\begin{notation} Let
	 \begin{equation}
	 P_k(l) := \prod\limits_{i=0}^{k-1}(1+\hbar(l+i+1/2)).
	 \end{equation}
In particular,
\begin{equation}
D^m(\hbar)\alpha_{-k}D^{-m}(\hbar) = \sum\limits_{l\in\Z+\frac{1}{2}} (P_k(l))^m E_{l+k,l}.
\end{equation}
\end{notation}

\begin{lemma}\label{L:conj2}
For any positive integer $a$
\begin{equation}
e^{\alpha_1} \sum\limits_{l\in\Z+\frac{1}{2}} f(l)E_{l+a,l} \, e^{-\alpha_1} = \sum\limits_{l\in\Z+\frac{1}{2}}\sum_{t=0}^\infty \frac{(\Delta^t\,f)(l)}{t!}E_{l+a-t,l} + \frac{(\Delta^{a-1} f)_{-1/2}}{a!}\mathrm{Id}.
\end{equation}
\end{lemma}

\begin{proof}
Recall Hadamard's formula: $e^XYe^{-X} = e^{\mathrm{ad}_X} (Y)$, where $\mathrm{ad}_X(\cdot) = [X,\cdot]$. Using this and Lemma \ref{l:Ecomm} we have:
\begin{align}
 e^{\alpha_1} \sum\limits_{l\in\Z+\frac{1}{2}} f(l)E_{l+a,l} \, e^{-\alpha_1} & 
 =\sum\limits_{l\in\Z+\frac{1}{2}}\sum_{t=0}^\infty \frac{(\Delta^t\,f)(l)}{t!}E_{l+a-t,l} 
 \\ \notag & \phantom{= {}}
 + (\delta_{a,1}f_{-1/2}+\frac{1}{2!}\delta_{a-1,1}(\Delta f)_{-1/2}+\frac{1}{3!}\delta_{a-2,1}(\Delta^2 f)_{-1/2}+\ldots)\mathrm{Id}
 \\ \notag &
 =  \sum\limits_{l\in\Z+\frac{1}{2}}\sum_{t=0}^\infty \frac{(\Delta^t\,f)(l)}{t!}E_{l+a-t,l} + \frac{(\Delta^{a-1} f)_{-1/2}}{a!}\mathrm{Id}.
\end{align}
\end{proof}
%
\begin{proposition} We have:
	\begin{align}
	\label{eq:checkA+Id}\mathcal{\check{A}}(k,\hbar) = \frac{\hbar^{-k}}{k}\left( \sum\limits_{l\in\Z+\frac{1}{2}}  \frac{(\Delta^{q+k}\,P^m_k)(l)}{(q+k)!} E_{l-q,l} + \frac{(\Delta^{k-1}\,P^m_k)(l)\vert_{l=-1/2}}{k!}\mathrm{Id}\right).
	\end{align}
\end{proposition}
\begin{proof}
	Follows from Definition \ref{def:acheck}, Lemma \ref{L:conj1} and Lemma \ref{L:conj2}. 
\end{proof}

\subsubsection{The difference operator and the $P$- and $R$-polynomials}

Now let us formulate some technical results about the difference operator $\Delta$.
\begin{lemma}\label{L:delta1}
If $t>k,$ then $\Delta^t P_k(l) =0$ and if $0\leq t\leq k$, then
\begin{equation}
\Delta^t P_k(l) = (k)_t \hbar^t \prod\limits_{i=0}^{k-t-1} (1+\hbar(l+i+1/2)) = (k)_t \hbar^t P_{k-t}(l),
\end{equation}
where we recall that $(k)_t$ denotes the falling factorial, see equation~\eqref{eq:defPoch}.
\end{lemma}
\begin{proof}
The proof is a straightforward induction on the parameter $t$.
\end{proof}

\begin{proposition}\label{prop:deltapexr} We have:
\begin{align} \label{EQ:P(l)}
\frac{\Delta^{t}P^m_k(l)}{t!}  & = \sum\limits_{i_1+\ldots+i_m=t} \prod\limits_{j=1}^{m} \frac{\Delta^{i_j}P_{k}(l-i_{j+1}-\ldots-i_m)}{i_j!}
\\ \notag &
=  \sum\limits_{i_1+\ldots+i_m=t} \hbar^t \prod\limits_{j=1}^{m} \frac{(k)_{i_j}}{i_j!} P_{k-i_j}(l-i_{j+1}-\ldots-i_m),
\end{align} 
where the summations go over all ordered tuples of non-negative integers $(i_1,\ldots,i_m).$
\end{proposition}
\begin{proof}
By induction on the parameter $t$ using Lemma \ref{L:delta1} and the following identity:
\begin{align}
(\Delta fg)(l)  = (\Delta f)(l)g(l) + f(l-1)(\Delta g)(l).
\end{align}
\end{proof}

In what follows we will extensively use the following expressions:
\begin{definition} We define
	\begin{align}	\label{eq:Rdef}
	R(\hbar;k,l,i_1,\dots,i_m)&:=\prod\limits_{j=1}^{m} P_{k-i_j}(l-i_{j+1}-\ldots-i_m)\\\nonumber &=\prod\limits_{j=1}^{m}  \prod\limits_{\gamma_j=0}^{k-i_j-1}(1+\hbar(l-i_{j+1}-\ldots -i_m+\gamma_j+1/2)); \\ 	\label{eq:Puexp}
	R_p(k,l,i_1,\dots,i_m)&:=[\hbar^p]\prod\limits_{j=1}^{m} P_{k-i_j}(l-i_{j+1}-\ldots-i_m) \\ \notag
	& =[\hbar^p]R(\hbar;k,l,i_1,\dots,i_m) .
	\end{align}
\end{definition}

\begin{proposition}	\label{prop:Rppoly}
	Expression
$R_p(k,l,i_1,\dots,i_m)$
	defined by equation~\eqref{eq:Puexp} is a polynomial of total degree $2p$ in all of its variables, and moreover $R_p(k,l,0,\dots,0)$ is divisible by $k$.
\end{proposition}
\begin{proof}
	We have 
	\begin{align}
	\label{Eq:degreep}
	&[\hbar^p]\prod\limits_{j=1}^{m} P_{k-i_j}(l-i_{j+1}-\ldots-i_m) = [\hbar^p]\prod\limits_{j=1}^{m}  \prod\limits_{\gamma_j=0}^{k-i_j-1}(1+\hbar(l-i_{j+1}-\ldots -i_m+\gamma_j+1/2))\\
	&=\sum_{\beta_1+\ldots+\beta_m=p}\;\prod_{j=1}^m\;\sum_{0\leq\gamma_1<\ldots<\gamma_{\beta_j}< k-i_j}\;\prod_{r=1}^{\beta_j} (l-i_{j+1}-\ldots -i_m+\gamma_r+1/2) \nonumber
	\end{align}
	From the last formula it is easy to see that the current proposition follows from Proposition \ref{prop:extFaulhaber}, including the divisibility statement.
\end{proof}

\subsubsection{The quasi-rationality statements}

\begin{proposition}\label{P:uEcoeff}
	For any $q\in \mathbb{Z}_{>0}$ and $p\in \mathbb{Z}_{\geq 0}$ we have
\begin{equation}\label{eq:AcoefsQ}
[\hbar^{q+p}][E_{l-q,l}]\check{\mathcal{A}}(k,\hbar) = \sum_{\substack {0\leq\sigma\leq \min(2p,\,q+k) \\ s_1+\dots+s_m=\sigma }}
Q^p_{s_1,\dots,s_m}(k,l)\dfrac{(k)_{s_1}\cdots(k)_{s_m}}{k(q+k-\sigma)!}\;(mk-\sigma)_{k+q-\sigma},
\end{equation}
where $Q^p_{s_1,\dots,s_m}(k,l)$ are some polynomials in $k$ and $l$. 
Moreover, the polynomials $Q^p_{0,\dots,0}(k,l)$ are divisible by $k$.
\end{proposition}
\begin{proof}
	Consider the following basis in the space of polynomials of total degree $\leq 2p$ in variables $i_1,\dots,i_m$:
	\begin{equation}
	\{(i_1)_{s_1}\cdots(i_m)_{s_m}|s_1+\ldots+s_m\leq 2p\}
	\end{equation}
	Since it follows from Proposition \ref{prop:Rppoly} (in the notation of that proposition) that  $R_p$ is a polynomial in $i_1,\dots,i_m$ of total degree $2p$, it can be expressed in terms of this basis:
	\begin{equation} \label{eq:RQexpan}
	R_p(k,l,i_1,\dots,i_m)= \sum_{\sigma=0}^{2p}\;\sum_{s_1+\dots+s_m=\sigma}Q^p_{s_1,\dots,s_m}(k,l)\; \cdot\;(i_1)_{s_1}\cdots(i_m)_{s_m},
	\end{equation}	
	where $Q^p_{s_1,\dots,s_m}(k,l)$ are some polynomials in $k$ and $l$ of degree $\leq 2p$. Since $R_p(k,l,0,\dots,0)$ is divisible by $k$ we naturally obtain that $Q^p_{0,\dots,0}(k,l)$ is divisible by $k$.
	
	Now, with the help of Proposition \ref{prop:deltapexr}, we can write
	\begin{align} \label{eq:AEcoef}
	&[\hbar^{q+p}][E_{l-q,l}]\check{\mathcal{A}}(k,\hbar)=[\hbar^{q+p}] \frac{\hbar^{-k}}{k} \frac{(\Delta^{q+k}\,P^m_k)(l)}{(q+k)!}\\ \nonumber
	&= 
	[\hbar^{q+p}]\sum\limits_{i_1+\ldots+i_m=q+k} \frac{\hbar^{q}}{k} \prod\limits_{j=1}^{m} \frac{(k)_{i_j}}{i_j!} P_{k-i_j}(l-i_{j+1}-\ldots-i_m) \\ \nonumber
	&=\frac{1}{k}\sum\limits_{i_1+\ldots+i_m=q+k}	R_p(k,l,i_1,\dots,i_m) \prod\limits_{j=1}^{m} \frac{(k)_{i_j}}{i_j!} \\ \nonumber
	&=\frac{1}{k}\sum\limits_{i_1+\ldots+i_m=q+k} \;	 \sum_{\sigma=0}^{2p}\;\sum_{s_1+\dots+s_m=\sigma}Q^p_{s_1,\dots,s_m}(k,l)\prod\limits_{j=1}^{m} \dfrac{(k)_{i_j}\cdot (i_j)_{s_j}}{i_j!}\\ \nonumber
	&= \frac{1}{k}\sum_{\substack {0\leq\sigma\leq \min(2p,\,q+k) \\s_1+\dots+s_m=\sigma}}Q^p_{s_1,\dots,s_m}(k,l)\cdot(k)_{s_1}\cdots(k)_{s_m}\sum\limits_{i_1+\ldots+i_m=q+k}\;\prod\limits_{j=1}^{m} \dfrac{(k-s_j)_{i_j-s_j}}{(i_j-s_j)!}\\ \nonumber
	&= \frac{1}{k}\sum_{\substack {0\leq\sigma\leq \min(2p,\,q+k) \\s_1+\dots+s_m=\sigma}}Q^p_{s_1,\dots,s_m}(k,l)\cdot(k)_{s_1}\cdots(k)_{s_m} \cdot \dfrac{(mk-\sigma)_{q+k-\sigma}}{(q+k-\sigma)!}
	\end{align}
	Note that in this computation in the third line from the bottom all terms with $\sigma>q+k$ vanish as in that case $\exists j:\; i_j<s_j$, and thus $(i_j)_{s_j}=0$ for that $j$.
	In the last equality we used the falling factorial version of the multinomial formula, i.e.
	\begin{equation}
	\sum\limits_{i_1+\ldots+i_m=t}\dfrac{t!}{i_1!\cdots i_m!}\, (k_1)_{i_1}\cdots (k_m)_{i_m}= (k_1+\ldots+k_m)_t
	\end{equation}
	This proves the proposition.
\end{proof}


\begin{proposition}\label{prop:AcoefsPoly} We have:
\begin{align}\label{eq:AcoefsPoly}
& [\hbar^{q+p}][E_{l-q,l}]\check{\mathcal{A}}(k,\hbar)
\\ \notag 
& =\dfrac{(mk-m)!}{k!\,(mk-k-1)!} \cdot \dfrac{1}{\displaystyle\prod_{\substack{m\leq j \leq 2p-1 \\  m\nmid j}} \left(k-\dfrac{j}{m}\right)}\cdot\dfrac{\mathrm{S}_{p,l,q}(k)}{(k+1)(k+2)\cdots(k+q)},
\end{align}
where $\mathrm{S}_{p,l,q}(k)$ is some polynomial in $k$ of degree $\leq 6p+q$.
\end{proposition}
\begin{proof} 
Let us look at equation \eqref{eq:AcoefsQ}.	For $q\geq \sigma$ and $\sigma \geq m$ we have
\begin{align}
\label{eq:prop4.6}
&\dfrac{(k)_{s_1}\cdots(k)_{s_m}}{(q+k-\sigma)!}\;(mk-\sigma)_{k+q-\sigma}\\ \nonumber
&=\dfrac{(mk-m)!}{k!\,(mk-k-1)!} \cdot (k)_{s_1}\cdots(k)_{s_m} \cdot \dfrac{1}{(k+1)\cdots(k+q-\sigma)}\\ \nonumber
&\phantom{= {}}
\times\dfrac{(mk-k-1)!}{(mk-\sigma+1)(mk-\sigma+2)\cdots(mk-m)}\cdot\dfrac{1}{(mk-k-q)!}\\ \nonumber
&=\dfrac{(mk-m)!}{k!\,(mk-k-1)!} \cdot (k)_{s_1}\cdots(k)_{s_m} \cdot \dfrac{1}{(k+1)\cdots(k+q-\sigma)}\\ \nonumber
&\phantom{= {}}
\times\dfrac{m^{m-\sigma}\,(mk-k-1)_{q-1}}{\displaystyle\prod_{\substack{m\leq j \leq \sigma-1 \\ m\nmid j}} \left(k-\dfrac{j}{m}\right)}\cdot \dfrac{1}{\displaystyle\prod_{1\leq r\leq \lfloor(\sigma-1)/m\rfloor} \left(k-r\right)}
\end{align}
Since $s_1+\ldots+s_m=\sigma$, for at least one $j$ we have $s_j\geq\sigma/m$. This means that for this~$j$ all product terms in $\prod_{1\leq r\leq \lfloor(\sigma-1)/m\rfloor} \left(k-r\right)$ get cancelled with terms from $(k)_{s_j}$.

Thus any term from the sum in the RHS of \eqref{eq:AcoefsQ} for $q\geq \sigma$ and $\sigma \geq m$ is in fact of the form as in the RHS of \eqref{eq:AcoefsPoly}. The cases when $q$,~$\sigma$ and $\sigma$,~$m$ are related in three other possible ways are almost completely analogous, only easier (as for e.g. $\sigma<m$ the extra poles do not even appear).

What remains is to see that the factor $1/k$ from \eqref{eq:AcoefsQ} gets canceled. For $\sigma>0$ at least one of the $s_j$ is greater than zero, and thus the product $(k)_{s_1}\cdots(k)_{s_m}$ is divisible by $k$ which cancels $1/k$. For $\sigma=0$ we have $\forall j\; s_j=0$, and thus this term comes with $Q^p_{0,\dots,0}(k,l)$ which is divisible by $k$ as stated in Proposition \ref{P:uEcoeff}, which, again, cancels the $1/k$ factor. 

Finally, the upper bound of $6p+q$ on the degree of $\mathrm{S}_{p,l,q}(k)$ comes from the fact that polynomials $Q^p_{s_1,\dots,s_m}$ have degree $\leq 2p$ and from comparing \eqref{eq:prop4.6} with  \eqref{eq:AcoefsPoly} and \eqref{eq:AcoefsQ}.
%
\end{proof}

\begin{proposition}\label{P:uEIDcoeff}
	For any  $p\in \mathbb{Z}_{\geq 0}$ we have:
	\begin{align}\label{eq:AIDcoefsQ}
	[\hbar^{p}][\Id]\check{\mathcal{A}}(k,u) = \sum_{\substack{0\leq\sigma\leq \min(2p+2,\,k-1) \\ s_1+\dots+s_m=\sigma}}Q^{p+1}_{s_1,\dots,s_m}(k,-1/2)\dfrac{(k)_{s_1}\cdots(k)_{s_m}}{k^2(k-1-\sigma)!}\;(mk-\sigma)_{k-1-\sigma},
	\end{align}
	where $Q^p_{s_1,\dots,s_m}(k,l)$ are the same polynomials in $k$ and $l$ as in Proposition \ref{P:uEcoeff}.
\end{proposition}
\begin{proof}
	
Recall equation \eqref{eq:RQexpan} and Proposition \ref{prop:deltapexr}. Analogous to the proof of Proposition \ref{P:uEcoeff}, we have
	\begin{align} \label{eq:AIDEcoef}
	&[\hbar^{p}][\Id]\check{\mathcal{A}}(k,\hbar)=[\hbar^{p}] \frac{\hbar^{-k}}{k} \frac{(\Delta^{k-1}\,P^m_k)(l)\vert_{l=-1/2}}{k!}\\ \nonumber
	&= 
	[\hbar^{p+1}]\sum\limits_{i_1+\ldots+i_m=k-1} \frac{1}{k^2} \prod\limits_{j=1}^{m} \frac{(k)_{i_j}}{i_j!} P_{k-i_j}(-1/2-i_{j+1}-\ldots-i_m) \\ \nonumber
	&=\frac{1}{k^2}\sum\limits_{i_1+\ldots+i_m=k-1}	R_{p+1}(k,-1/2,i_1,\dots,i_m) \prod\limits_{j=1}^{m} \frac{(k)_{i_j}}{i_j!} \\ \nonumber
	&=\frac{1}{k^2}\sum\limits_{i_1+\ldots+i_m=k-1} \;	 \sum_{\sigma=0}^{2p+2}\;\sum_{s_1+\dots+s_m=\sigma}Q^{p+1}_{s_1,\dots,s_m}(k,-1/2)\prod\limits_{j=1}^{m} \dfrac{(k)_{i_j}\cdot (i_j)_{s_j}}{i_j!}\\ \nonumber
	&= \frac{1}{k^2}\sum_{\substack{0\leq\sigma\leq \min(2p+2,\,k-1) \\ s_1+\dots+s_m=\sigma}}Q^{p+1}_{s_1,\dots,s_m}(k,-1/2)\cdot(k)_{s_1}\cdots(k)_{s_m}\sum\limits_{i_1+\ldots+i_m=k-1}\;\prod\limits_{j=1}^{m} \dfrac{(k-s_j)_{i_j-s_j}}{(i_j-s_j)!}\\ \nonumber
	&= \frac{1}{k^2}\sum_{\substack{0\leq\sigma\leq \min(2p+2,\,k-1) \\ s_1+\dots+s_m=\sigma}}Q^{p+1}_{s_1,\dots,s_m}(k,-1/2)\cdot(k)_{s_1}\cdots(k)_{s_m} \cdot \dfrac{(mk-\sigma)_{k-1-\sigma}}{(k-1-\sigma)!}.
	\end{align}
	This proves the proposition.
\end{proof}


\begin{proposition}\label{prop:AIDcoefsPoly} We have:
	\begin{equation}\label{eq:AIDcoefsPoly}
	[\hbar^{p}][\Id]\check{\mathcal{A}}(k,\hbar)=\dfrac{(mk-m)!}{k!\,(mk-k-1)!} \cdot \dfrac{1}{\displaystyle\prod_{\substack{m\leq j \leq 2p+1 \\ m\nmid j}} \left(k-\dfrac{j}{m}\right)}\cdot\dfrac{\mathrm{S}^{\Id}_{p}(k)}{k^2\, (mk-k+1)},
	\end{equation}
	where $\mathrm{S}^{\Id}_{p}(k)$ is some polynomial in $k$.
\end{proposition}
\begin{proof} 
	Let us look at equation \eqref{eq:AIDcoefsQ}. For $\sigma\geq m$ we have
	\begin{align}
	\label{eq:IDprop4.6}
	&\dfrac{(k)_{s_1}\cdots(k)_{s_m}}{(k-1-\sigma)!}\;(mk-\sigma)_{k-1-\sigma}\\ \nonumber
	&=\dfrac{(mk-m)!}{k!\,(mk-k-1)!} \cdot (k)_{s_1}\cdots(k)_{s_m} \cdot (k-\sigma)\cdots (k-1)\, k\\ \nonumber
	&\phantom{= {}}\times\dfrac{(mk-k-1)!}{(mk-\sigma+1)(mk-\sigma+2)\cdots(mk-m)}\cdot\dfrac{1}{(mk-k+1)!}\\ \nonumber
	&=\dfrac{(mk-m)!}{k!\,(mk-k-1)!} \cdot (k)_{s_1}\cdots(k)_{s_m} \cdot (k-\sigma)\cdots (k-1)\, k\\ \nonumber
	&\phantom{= {}}\times\dfrac{m^{m-\sigma}}{\displaystyle\prod_{\substack{m\leq j \leq \sigma-1 \\ m\nmid j}} \left(k-\dfrac{j}{m}\right)}\cdot \dfrac{1}{\displaystyle\prod_{1\leq r\leq \lfloor(\sigma-1)/m\rfloor} \left(k-r\right)}\cdot\dfrac{1}{(mk-k)\,(mk-k+1)}\\ \nonumber
	&=\dfrac{(mk-m)!}{k!\,(mk-k-1)!} \cdot (k)_{s_1}\cdots(k)_{s_m} \\ \nonumber
	&\phantom{= {}}\times\dfrac{m^{m-\sigma}}{\displaystyle\prod_{\substack{m\leq j \leq \sigma-1 \\ m\nmid j}} \left(k-\dfrac{j}{m}\right)}\cdot \left(\displaystyle\prod_{\lfloor(\sigma-1)/m\rfloor<r\leq\sigma} \left(k-r\right)\right)\cdot\dfrac{1}{(m-1)\,(mk-k+1)}
	\end{align}
	Thus any term from the sum in the RHS of \eqref{eq:AIDcoefsQ} for $\sigma \geq m$ is in fact of the form as in the RHS of \eqref{eq:AIDcoefsPoly}.

	The case when $\sigma<m$ is almost completely analogous, only easier (as in that case the extra poles at $j/m$ do not even appear). This proves the proposition.
\end{proof}

\subsubsection{Further technical statements}
In order to continue we need a further couple of technical propositions.
\begin{proposition}\label{prop:DeltaRAct}
	We have:
	\begin{align}\label{eq:DeltaR}
	&
	\Delta_{i_j} R(\hbar;k,-1/2,i_1,\dots,i_m) 
	\\ \nonumber &
	=\left(1-(1+\hbar(k-i_j-\dots-i_m))\cdot\prod_{q=1}^{j-1}\dfrac{1+\hbar(k-i_q-\dots-i_m)}{1+\hbar(-i_{q+1}-\dots-i_m)}\right)\\ \notag
	& \phantom{= {}} \times R(\hbar;k,-1/2,i_1,\dots,i_m),
	\end{align}
	where $\Delta_{i_j}$ is the backward difference operator in variable $i_j$ which acts on a function $f(i_1,\dots,i_m)$ in the following way:
	\begin{equation}
	\Delta_{i_j} f(i_1,\dots,i_m) \coloneqq f(i_1,\dots,i_{j-1},i_j,i_{j+1},\dots,i_m) - f(i_1,\dots,i_{j-1},i_j-1,i_{j+1},\dots,i_m).
	\end{equation}
\end{proposition}
\begin{proof}
	From the definition of $R(\hbar;k,l,i_1,\dots,i_m)$, i.e.  formula \eqref{eq:Rdef}, we have
	\begin{equation}
	R(\hbar;k,-1/2,i_1,\dots,i_m)=\prod\limits_{n=1}^{m}  \prod\limits_{\gamma_n=0}^{k-i_n-1}(1+\hbar(\gamma_n-i_{n+1}-\ldots -i_m))
	\end{equation}
	From the definition of $\Delta_{i_j}$ we have
	\begin{align}
	\Delta_{i_j} R(\hbar;k,-1/2,i_1,\dots,i_m) & = R(\hbar;k,-1/2,i_1,\dots,i_j,\dots,i_m)
	\\ \nonumber
	&\phantom{= {}} -R(\hbar;k,-1/2,i_1,\dots,i_j-1,\dots,i_m).
	\end{align}
	Note that the polynomial $R(\hbar;k,-1/2,i_1,\dots,i_j-1,\dots,i_m)$ has mostly the same factors as $R(\hbar;k,-1/2,i_1,\dots,i_j,\dots,i_m)$, apart from an extra factor of $(1+\hbar(k-i_j-\dots-i_m))$ coming from the increased upper product limit in the $j$-th product, and also the factors 
	\begin{equation}
	\dfrac{1+\hbar(k-i_q-\dots-i_m)}{1+\hbar(-i_{q+1}-\dots-i_m)}
	\end{equation}
	for each $q$-th product for $q=1,\dots,j-1$, since shifting $i_j$ by $-1$ in the factors where it is present is the same as shifting both lower and upper product limits by $+1$ in the corresponding products. Thus we obtain \eqref{eq:DeltaR}, which proves the proposition.
\end{proof}

\begin{proposition}\label{prop:EulerRAct}
	We have
	\begin{equation}\label{eq:ER}
	\big(\mathfrak{E}\, R(\hbar;k,-1/2,i_1,\dots,i_m) \big)\Big|_{i_1=\dots=i_m=k} =  \binom{m}{2}\; k^2 \hbar,
	\end{equation}
	where $\mathfrak{E}$ is the discrete Euler operator in the $i$-variables defined as
	\begin{equation}\label{eq:EulerI}
	\mathfrak{E}:=\sum_{j=1}^m i_j\, \Delta_{i_j}.
	\end{equation}
\end{proposition}
\begin{proof}
	From Proposition \ref{prop:DeltaRAct} we have 
	\begin{align}
	&\mathfrak{E}\, R(\hbar;k,-1/2,i_1,\dots,i_m)\\ \nonumber
	&=\left(\sum_{j=1}^m i_j \left(1-(1+\hbar(k-i_j-\dots-i_m))\cdot\prod_{q=1}^{j-1}\dfrac{1+\hbar(k-i_q-\dots-i_m)}{1+\hbar(-i_{q+1}-\dots-i_m)}\right)\right) \\ \nonumber &\phantom{= {}}\times R(\hbar;k,-1/2,i_1,\dots,i_m)
	\end{align}
	Once we plug $i_j=k$ for every $j$ we get
	\begin{align}
	&\big(\mathfrak{E}\, R(\hbar;k,-1/2,i_1,\dots,i_m) \big)\Big|_{i_1=\dots=i_m=k}\\ \nonumber
	&=\left(\sum_{j=1}^m k \left(1-(1+\hbar(j-m)k)\cdot\prod_{q=1}^{j-1}\dfrac{1+\hbar(q-m)k}{1+\hbar(q-m)k}\right)\right) R(\hbar;k,-1/2,k,\dots,k) \\ \nonumber
	&=\left(\sum_{j=1}^m k^2 \hbar (m-j)\right) R(\hbar;k,-1/2,k,\dots,k)\\ \nonumber
	&=\dfrac{m(m-1)}{2}\; k^2 \hbar\; R(\hbar;k,-1/2,k,\dots,k)\\ \nonumber
	&=\binom{m}{2}\; k^2 \hbar
	\end{align}
	The last equality holds since $R(\hbar;k,-1/2,k,\dots,k)=1$ as for these values of its arguments $R$ is a product of zero terms and thus equal to $1$.
\end{proof}

\begin{proposition}\label{prop:IdDivis}
	For $p\geq 0$ the polynomial $\mathrm{S}^{\Id}_p(k)$ from Proposition \ref{prop:AIDcoefsPoly} is divisible by $k^2$ and by $(mk-k+1)$.
\end{proposition}
\begin{proof}
	Recall \eqref{eq:RQexpan}, which is the definition of the $Q$-polynomials:
	\begin{equation}
	R_p(k,l,i_1,\dots,i_m)= \sum_{\sigma=0}^{2p}\;\sum_{s_1+\dots+s_m=\sigma}Q^p_{s_1,\dots,s_m}(k,l)\; \cdot\;(i_1)_{s_1}\cdots(i_m)_{s_m}.
	\end{equation}	
	Denote	
	\begin{equation}
	V_p^{(\sigma)}(k,l,i_1,\dots,i_m):= \sum_{s_1+\dots+s_m=\sigma}Q^p_{s_1,\dots,s_m}(k,l)\; \cdot\;(i_1)_{s_1}\cdots(i_m)_{s_m},
	\end{equation}
	so that
	\begin{equation}
	R_p(k,l,i_1,\dots,i_m) =  \sum_{\sigma=0}^{2p} V_p^{(\sigma)}(k,l,i_1,\dots,i_m).
	\end{equation}
%
	Introduce the following rational expression in $k$:
	\begin{equation} \label{eq:rhodef}	
	\rho_p(k) :=  \sum_{\sigma=0}^{2p} \dfrac{(k-1)_\sigma}{(mk)_\sigma} \;V_p^{(\sigma)}(k,-1/2,k,\dots,k).
	\end{equation}
%
	Note that 
	\begin{equation}\label{eq:Arhorel}
	[\hbar^{p-1}][\Id]\check{\mathcal{A}}(k,\hbar)=\dfrac{1}{k^2}\;\rho_p(k)\; \dfrac{(mk)_{k-1}}{(k-1)!},
	\end{equation}
	as follows from \eqref{eq:AIDcoefsQ} and \eqref{eq:RQexpan}. We also have
	\begin{equation}\label{eq:coefsrel}
	\dfrac{(mk)_{k-1}}{(k-1)!} = \dfrac{(mk-m)!}{k!\,(mk-k-1)!} \cdot \dfrac{mk (mk-1)(mk-2)\cdots(mk-m+1)}{(m-1)(mk-k+1)}.
	\end{equation}
	Equations \eqref{eq:AIDcoefsPoly}, \eqref{eq:Arhorel}, and \eqref{eq:coefsrel} together imply that
	\begin{equation}
	\mathrm{S}^{\Id}_{p-1}(k)= \rho_p(k)\; \cdot\; k\;\cdot\;  \dfrac{m (mk-1)(mk-2)\cdots(mk-m+1)}{m-1}\cdot \prod_{\substack{m\leq j \leq 2p-1 \\ m\nmid j}} \left(k-\dfrac{j}{m}\right).
	\end{equation}
	From the last equation it is evident that in order to prove the present proposition it is sufficient to prove that for $p>0$ expression $\rho_p(k)$
	has zeroes at $k=1/(1-m)$, i.e.~at $mk-k+1=0$, and at $k=0$. Let us prove it.
	
	
	First, let us prove that
	\begin{equation} \label{eq:trhoop}
	\rho_p(k) = \left.\left(\dfrac{(mk-\mathfrak{E})_{mk-k+1}}{(mk)_{mk-k+1}}\; R_p(k,-1/2,i_1,\dots,i_m)\right)\right|_{i_1=\dots=i_m=k},
	\end{equation}
	where $\mathfrak{E}$ is the discrete Euler operator in the $i$-variables defined in \eqref{eq:EulerI}.
	In order to prove this equation note that 
	\begin{equation}
	\mathfrak{E} \; (i_1)_{s_1}\cdots(i_m)_{s_m} = (s_1+\dots+s_m) \,\cdot\,(i_1)_{s_1}\cdots(i_m)_{s_m}
	\end{equation}
	and that 
	\begin{equation}
	\dfrac{(k-1)_\sigma}{(mk)_\sigma} = \dfrac{\Gamma(k)\;\Gamma(mk+1-\sigma)}{\Gamma(k-\sigma)\;\Gamma(mk+1)} = \dfrac{(mk-\sigma)_{mk-k+1}}{(mk)_{mk-k+1}}.
	\end{equation}
	Thus (for $s_1+\dots+s_m=\sigma$) we have
	\begin{equation}
	\dfrac{(k-1)_\sigma}{(mk)_\sigma}\; (k)_{s_1}\cdots(k)_{s_m} = \left.\left(\dfrac{(mk-\mathfrak{E})_{mk-k+1}}{(mk)_{mk-k+1}}\; (i_1)_{s_1}\cdots(i_m)_{s_m}\right)\right|_{i_1=\dots=i_m=k},
	\end{equation}
	which proves \eqref{eq:trhoop} after one takes into account the definition of $\rho_p(k)$, i.e. formula \eqref{eq:rhodef}.

	
	Let us prove that $\rho_p(k)$ for $p>0$ has a zero at $mk-k+1=0$, i.e. at $k=1/(1-m)$. Let $mk-k+1=0$. In that case we have
	\begin{align} 
	&\left.\rho_p(k)\right|_{mk-k+1=0} = \left.\left(\dfrac{(mk-\mathfrak{E})_{mk-k+1}}{(mk)_{mk-k+1}}\; R_p(k,-1/2,i_1,\dots,i_m)\right)\right|_{\substack{i_1=\dots=i_m=k\\mk-k+1=0}} \\ \nonumber
	&=\left(R_p(k,-1/2,i_1,\dots,i_m)\right)\bigg|_{\substack{i_1=\dots=i_m=k\\mk-k+1=0}} = \left(R_p(k,-1/2,k,\dots,k)\right)\big|_{mk-k+1=0} \\ \nonumber
	&=\delta_{p,0} \big|_{mk-k+1=0} = \delta_{p,0}.
	\end{align}
	
	What is left is to prove that $\rho_p(k)$ has a zero at $k=0$.
%
	Let us introduce the following function that depends on $k$ and on a new formal variable $\tilde{k}$:
	\begin{align} 
	\widetilde{\rho}_p(k,\tilde{k}) &:= \sum_{\sigma=0}^{2p} \dfrac{(m(k-\tilde{k})+\tilde{k}-1)_\sigma}{(mk)_\sigma} \;V_p^{(\sigma)}(k,-1/2,k,\dots,k)\\ \nonumber  &=  \left.\left(\dfrac{(mk-\mathfrak{E})_{m\tilde{k}-\tilde{k}+1}}{(mk)_{m\tilde{k}-\tilde{k}+1}}\; R_p(k,-1/2,i_1,\dots,i_m)\right)\right|_{i_1=\dots=i_m=k}.
	\end{align}
	Note that 
	\begin{equation}\label{eq:rhotform}
	\widetilde{\rho}_p(k,\tilde{k}) = \dfrac{\mathfrak{P}(k,\tilde{k})}{\mathfrak{Q}(k)},
	\end{equation}
	where $\mathfrak{P}(k,\tilde{k})$ is a polynomial in $k$ and $\tilde{k}$ and $\mathfrak{Q}(k)$ is a polynomial in $k$ not divisible by $k$. The fact that $\mathfrak{P}$ and $\mathfrak{Q}$ are polynomials is evident from the definition, while the fact that $\mathfrak{Q}$ does not have a zero at $k=0$ requires a bit of explanation. Note that the poles of the rational function $\widetilde{\rho}$ come only from $(mk)_\sigma$ in the denominators of the terms in the sum in the definition of $\widetilde{\rho}$. For $\sigma>0$ expression $(mk)_\sigma$ has a simple zero at $k=0$. But this single factor of $k$ gets canceled with $k$ coming from one of the $(k)_{s_j}$ factors (since for $\sigma>0$ there is at least one $s_j>0$). Thus terms with $\sigma>0$ in $\rho$ do not have poles at $k=0$. The term with $\sigma=0$ does not have any poles at all, which proves the statement.
	This means that the $\widetilde{\rho}_p(k,\tilde{k})$ is continuous at $(k,\tilde{k})=(0,0)$ as a function of two variables.
	
	Note that 
	\begin{equation}\label{eq:rhorhot}
	\rho_p(k)=\widetilde{\rho}_p(k,k).
	\end{equation}
	Consider $\widetilde{\rho}_p(k,0)$:
	\begin{align} \label{eq:rhotexpr}
	\widetilde{\rho}_p(k,0) &= \left.\left(\dfrac{(mk-\mathfrak{E})_1}{(mk)_1}\; R_p(k,-1/2,i_1,\dots,i_m)\right)\right|_{i_1=\dots=i_m=k} \\ \nonumber
	&=\left.\left(\dfrac{mk-\mathfrak{E}}{mk}\; R_p(k,-1/2,i_1,\dots,i_m)\right)\right|_{i_1=\dots=i_m=k}\\ \nonumber
	&=[\hbar^p]\left.\left(\dfrac{mk-\mathfrak{E}}{mk}\; R(k,-1/2,i_1,\dots,i_m)\right)\right|_{i_1=\dots=i_m=k}\\ \nonumber
	&= [\hbar^p]\,\dfrac{mk-\binom{m}{2}\, k^2 \hbar}{mk}.
	\end{align}
	In this computation in the last equality we used the result of Proposition \ref{prop:EulerRAct} and the fact that $R(k,-1/2,k,\dots,k)=1$.
	Since our rational expression $\widetilde{\rho}_p(k,\tilde{k})$ is continuous at $(k,\tilde{k})=(0,0)$ and since $\rho_p(k)=\widetilde{\rho}_p(k,k)$, we have:
	\begin{equation}
	\rho_p(0) = \widetilde{\rho}_p(k,\tilde{k})|_{k=\tilde{k}=0} = \widetilde{\rho}_p(k,0)|_{k=0}=\left([\hbar^p]\,\left(1-\dfrac{m-1}{2}\,k\hbar\right)\right)\bigg|_{k=0}=[\hbar^p] 1= \delta_{p,0},
	\end{equation}
	which implies that $\rho_p(0)=0$ for $p>0$. This proves the proposition.
\end{proof}

\subsection{The $\check{\mathcal{A}}^\dagger$-operators}

In the present subsection we introduce and study the $\check{\mathcal{A}}^\dagger$-operators and relate them to the residues at negative integer points of the rescaled $\check{\mathcal{A}}$-operators themselves.

\subsubsection{Definition of $\check{\mathcal{A}}^\dagger$ }
Let us define the $\check{\mathcal{A}}^\dagger$-operator as
\begin{equation}\label{eq:ainvopdef}
\check{\mathcal{A}}^{\dagger}(k,\hbar) \coloneqq \hbar^k e^{\alpha_1}D(\hbar)^m\frac{\alpha_{k}}{k}D(\hbar)^{-m}e^{-\alpha_1}.
\end{equation}
We want to express it in terms of functions $P_k(l).$

\begin{notation}
	Denote
	\begin{equation}
	\tilde{P}_k(l) := \prod\limits_{i=0}^{k-1}(1-\hbar(-l+i+1/2))
	\end{equation}
\end{notation}

\begin{lemma}
For a positive integer $k$
\begin{equation}
D^m(\hbar)\alpha_{k}D^{-m}(\hbar) =\sum\limits_{l\in\Z+\frac{1}{2}} \tilde{P}^{-m}_k(l) E_{l-k,l}.
\end{equation}
\end{lemma}
\begin{proof}
The proof is analogous to the proof of Lemma \ref{L:conj1}.
\end{proof}

Therefore using Lemma \ref{L:conj2} we conclude that
\begin{equation}
[E_{l-k-t,l}]\check{\mathcal{A}}^{-1}(k,\hbar) =  \frac{\hbar^k}{k}\frac{(\Delta^t \tilde{P}^{-m}_k)(l)}{t!}.
\end{equation}

\subsubsection{Technical statements about the $\tilde P$- and $\tilde R$-polynomials}
Similarly to Lemma \ref{L:delta1} and Proposition \ref{prop:deltapexr} we get the following identity for the polynomials $\tilde{P}_k(l)$:
If $t>k,$ then $\Delta^t \tilde{P}_k(l) =0$ and if $0\leq t\leq k$, then
\begin{equation}
\Delta^t \tilde{P}^{-1}_k(l) = (-1)^t k^{(t)} \hbar^t \tilde{P}^{-1}_{k+t}(l),
\end{equation}
where $k^{(t)} = k(k+1)\ldots(k+t-1)$ is the rising factorial. The next identity holds for any function $F(l)$:
\begin{equation}
\frac{(\Delta^{t}F^m)(l)}{t!}  = \sum\limits_{i_1+\ldots+i_m=t} \prod\limits_{j=1}^{m} \frac{\Delta^{i_j}}{i_j!}F(l-i_{j+1}-\ldots-i_m)
\end{equation}
Using the last two formulae we get
\begin{equation}
\frac{(\Delta^{t} \tilde{P}^{-m}_k)(l)}{t!}  = 
 \sum\limits_{i_1+\ldots+i_m=t} (-1)^t\hbar^t \prod\limits_{j=1}^{m} \frac{k^{(i_j)}}{i_j!} \tilde{P}^{-1}_{k+i_j}(l-i_{j+1}-\ldots-i_m),
\end{equation}
where the summations go over all ordered tuples of nonnegative integers $(i_1,\ldots,i_m).$

We will need the following two technical propositions:

\begin{proposition}
	For
	\begin{equation} 
	\tilde{R}_p(k,l,i_1,\dots,i_m):=[\hbar^p]\prod\limits_{j=1}^{m} \tilde{P}^{-1}_{k+i_j}(l-i_{j+1}-\ldots-i_m)
	\end{equation}
	we have:
	\begin{equation} \label{eq:RtRmanyrel}
	\tilde{R}_p(k,l,i_1,\dots,i_m)=R_p(-k,l,i_1,\dots,i_m)
	\end{equation}
	Note that $R_p(-k,l,i_1,\dots,i_m)$ is understood to be the result of substitution of $-k$ in place of $k$ in the polynomial $R_p(k,l,i_1,\dots,i_m)$ defined in \eqref{eq:Puexp}, which makes sense since it is a polynomial.
\end{proposition}
\begin{proof}
		First let us prove the proposition for the case $m=1$, i.e. that 
		\begin{equation} \label{eq:RtRrel}
		\tilde{R}_p(k,l,i)=R_p(-k,l,i).
		\end{equation}
		We have
		\begin{align} \label{eq:Rexpr}
		&\tilde{R}_p(k,l,i) = [\hbar^p]\tilde{P}^{-1}_{k+i}(l) = [\hbar^p]\prod\limits_{\gamma=0}^{k+i-1}(1-\hbar(-l+\gamma+1/2))^{-1} \\ \nonumber
		&=[\hbar^p]\prod\limits_{\gamma=1}^{k+i}(1+\hbar(-l+\gamma-1/2)+\hbar^2(-l+\gamma-1/2)^2+\ldots)\\ \nonumber
		&=\sum\limits_{1\leq\gamma_1\leq\ldots\leq\gamma_p\leq k+i}(-l+\gamma_1-1/2)\cdots(-l+\gamma_p-1/2)
		\end{align}
		and		
		\begin{align}  \label{eq:Rtexpr}
		&R_p(k,l,i) = [\hbar^p]P_{k-i}(l) = [\hbar^p]\prod\limits_{\gamma=0}^{k-i-1}(1+\hbar(l+\gamma+1/2)) \\ \nonumber
		&=\sum\limits_{0<\gamma_1 < \ldots <\gamma_p < k-i}(l+\gamma_1+1/2)\cdots(l+\gamma_p+1/2).
		\end{align}
		From \eqref{eq:Rexpr} and \eqref{eq:Rtexpr} it is easy to see that \eqref{eq:RtRrel} is directly implied by Proposition \ref{prop:TtTrel}.
		
		For general $m$ we have
		\begin{align} \label{eq:Rtmanyexpr}
		&\tilde{R}_p(k,l,i_1,\dots,i_m)=[\hbar^p]\prod\limits_{j=1}^{m}\tilde{P}^{-1}_{k+i_j}(l-i_{j+1}-\ldots-i_m)\\ \nonumber
		&=\sum_{\beta_1+\ldots+\beta_m=p}\;\prod\limits_{j=1}^{m} [\hbar^{\beta_j}]\tilde{P}^{-1}_{k+i_j}(l-i_{j+1}-\ldots-i_m)\\ \nonumber
		&=\sum_{\beta_1+\ldots+\beta_m=p}\;\prod\limits_{j=1}^{m} \wt{R}_{\beta_j}(k,l-i_{j+1}-\ldots-i_m,i_j)\\ \nonumber
		&\mathop{=}^{\eqref{eq:RtRrel}}\sum_{\beta_1+\ldots+\beta_m=p}\;\prod\limits_{j=1}^{m} R_{\beta_j}(-k,l-i_{j+1}-\ldots-i_m,i_j).
		\end{align}
		On the other hand,
		\begin{align} \label{eq:Rmanyexpr}
		R_p(k,l,i_1,\dots,i_m) & =[\hbar^p]\prod\limits_{j=1}^{m}P_{k-i_j}(l-i_{j+1}-\ldots-i_m)\\ \nonumber
		&=\sum_{\beta_1+\ldots+\beta_m=p}\;\prod\limits_{j=1}^{m} [\hbar^{\beta_j}]P_{k-i_j}(l-i_{j+1}-\ldots-i_m)\\ \nonumber
		&=\sum_{\beta_1+\ldots+\beta_m=p}\;\prod\limits_{j=1}^{m} R_{\beta_j}(k,l-i_{j+1}-\ldots-i_m,i_j).
		\end{align}
		We see that the last line of \eqref{eq:Rmanyexpr} turns precisely into the last line of \eqref{eq:Rtmanyexpr} once one plugs $-k$ in place of $k$, which proves the proposition.
\end{proof}




\subsubsection{The $\mathcal{A}$-operators and their residues}

\begin{definition} \label{def:Aoperdef}
	Denote
\begin{equation}\label{eq:Aoperdef}
	\mathcal{A}(k,\hbar):= \left(\dfrac{(mk-m)!}{k!\,(mk-k-1)!} \right)^{-1}\check{\mathcal{A}}(k,\hbar),
\end{equation}
or, in other words, 
\begin{equation}
\check{\mathcal{A}}(k,\hbar)= \dfrac{(mk-m)!}{k!\,(mk-k-1)!}\;\mathcal{A}(k,\hbar).
\end{equation}
\end{definition}

\begin{proposition}\label{prop:AinvcoefsPoly}
	For $r\in\mathbb{Z}_{>0}$ we have
	\begin{equation}\label{eq:AinvcoefsPoly}
\mathop{\Res}_{k=-r}\mathcal{A}(k,u) = c(r) \check{\mathcal{A}}^{\dagger}(r,u),
	\end{equation}
where $c(r)$ is a coefficient which depends only on $r$.
\end{proposition}
\begin{proof}
	Equation \eqref{eq:AEcoef} for $\mathcal{A}^{\dagger}$ takes the form
	\begin{align}\label{eq:Ainvcoefs}
	&[\hbar^{q+p}][E_{l-q,l}]\mathcal{A}^{\dagger}(k,\hbar)=[\hbar^{q+p}] \frac{\hbar^{k}}{k} \frac{(\Delta^{q-k}\,\tilde{P}^{-m}_k)(l)}{(q-k)!}\\ \nonumber
	&= 
	(-1)^{q-k}[\hbar^{q+p}]\sum\limits_{i_1+\ldots+i_m=q-k} \frac{\hbar^{q}}{k} \prod\limits_{j=1}^{m} \frac{k^{(i_j)}}{i_j!} \tilde{P}^{-1}_{k+i_j}(l-i_{j+1}-\ldots-i_m) \\ \nonumber
	&=
	(-1)^{q-k}\sum\limits_{i_1+\ldots+i_m=q-k}	\tilde{R}_p(k,l,i_1,\dots,i_m) \frac{1}{k} \prod\limits_{j=1}^{m} \frac{k^{(i_j)}}{i_j!} \\ \nonumber
	&\mathop{=}^{\eqref{eq:RtRmanyrel}}
	(-1)^{q-k}\sum\limits_{i_1+\ldots+i_m=q-k}	R_p(-k,l,i_1,\dots,i_m) \frac{1}{k} \prod\limits_{j=1}^{m} \frac{k^{(i_j)}}{i_j!} \\ \nonumber
	&=
	(-1)^{q-k}\sum\limits_{i_1+\ldots+i_m=q-k} \;	 \sum_{\sigma=0}^{2p}\;\sum_{s_1+\dots+s_m=\sigma}Q^p_{s_1,\dots,s_m}(-k,l) \frac{1}{k} \prod\limits_{j=1}^{m} \dfrac{k^{(i_j)}\cdot (i_j)_{s_j}}{i_j!}\\ \nonumber
	&= 
	(-1)^{q-k}\sum_{\sigma=0}^{2p}\;\sum_{s_1+\dots+s_m=\sigma}Q^p_{s_1,\dots,s_m}(-k,l)\cdot k^{(s_1)}\cdots k^{(s_m)} \frac{1}{k} \sum\limits_{i_1+\ldots+i_m=q-k}\;\prod\limits_{j=1}^{m} \dfrac{(k+s_j)^{(i_j-s_j)}}{(i_j-s_j)!}\\ \nonumber
	&= 
	(-1)^{q-k}\sum_{\sigma=0}^{2p}\;\sum_{s_1+\dots+s_m=\sigma}Q^p_{s_1,\dots,s_m}(-k,l)\cdot k^{(s_1)}\cdots k^{(s_m)} \cdot \dfrac{(mk+\sigma)^{(q-k-\sigma)}}{k(q-k-\sigma)!}.
	\end{align}
	We have 
	\begin{align} 
	&\mathop{\Res}_{k=-r}[\hbar^{q+p}][E_{l-q,l}]\mathcal{A}(k,\hbar)\\ \nonumber
	&=\mathop{\Res}_{k=-r} \left(\dfrac{(mk-m)!}{k!\,(mk-k-1)!} \right)^{-1}\sum_{\sigma=0}^{2p}\;\sum_{s_1+\dots+s_m=\sigma}Q^p_{s_1,\dots,s_m}(k,l)\cdot(k)_{s_1}\cdots(k)_{s_m} \cdot \dfrac{(mk-\sigma)_{q+k-\sigma}}{k(q+k-\sigma)!}.
	\end{align}
	
	Fix $r\leq q-\sigma$. For $q\geq \sigma$ and $\sigma \geq m$ we have:
	\begin{align}
	&\dfrac{(mk-\sigma)_{k+q-\sigma}}{(q+k-\sigma)!} 
	=\dfrac{(mk-m)!}{k!\,(mk-k-1)!}\cdot \dfrac{1}{(k+1)\cdots(k+q-\sigma)}\cdot\dfrac{(mk-k-1)_{q-1}}{(mk-\sigma+1)
		\cdots(mk-m)}.
	\end{align}
	Then
	\begin{align}
	&\mathop{\Res}_{k=-r}\left( \dfrac{Q^p_{s_1,\dots,s_m}(k,l)\cdot(k)_{s_1}\cdots(k)_{s_m}}{k(k+1)\cdots(k+q-\sigma)}\cdot\dfrac{(mk-k-1)_{q-1}}{(mk-\sigma+1)(mk-\sigma+2)\cdots(mk-m)}\right) \\ \nonumber
	&=\dfrac{(-1)^{m+r+q}}{r!} \dfrac{Q^p_{s_1,\dots,s_m}(-r,l)\cdot r^{(s_1)}\cdots r^{(s_m)}}{(q-\sigma-r)!}\cdot\dfrac{(mr-r+1)^{(q-1)}}{(mr+\sigma-1)(mr+\sigma-2)\cdots(mr+m)}\\ \nonumber
	&=\dfrac{(-1)^{m} (mr-r+1)^{(m+r-1)}}{r!} \cdot (-1)^{q-r}\;Q^p_{s_1,\dots,s_m}(-r,l)\cdot r^{(s_1)}\cdots r^{(s_m)}\dfrac{(mr+\sigma)^{(q-r-\sigma)}}{(q-r-\sigma)!}.
	\end{align}
	
	For all other possible variants of how $q$, $\sigma$ and $m$ are positioned with respect to each other the intermediate computations are slightly different (though analogous), but the final line is one and the same. Thus we have 
	\begin{align} \label{eq:AcAirel}
	&\mathop{\Res}_{k=-r}[\hbar^{q+p}][E_{l-q,l}]\mathcal{A}(k,\hbar)\\ \nonumber
	&=\dfrac{(-1)^{m} (mr-r+1)^{(m+r-1)}}{r!}   \\ \nonumber
	&\phantom{= {}}\times(-1)^{q-r}\sum_{\sigma=0}^{2p}\;\sum_{s_1+\dots+s_m=\sigma}Q^p_{s_1,\dots,s_m}(-r,l)\cdot r^{(s_1)}\cdots r^{(s_m)}\dfrac{(mr+\sigma)^{(q-r-\sigma)}}{(q-r-\sigma)!} \\ \nonumber
	&\mathop{=}^{\eqref{eq:Ainvcoefs}}\dfrac{(-1)^{m} (mr-r+1)^{(m+r-1)}}{r!}\;\cdot\;[\hbar^{q+p}][E_{l-q,l}]\mathcal{A}^{\dagger}(r,\hbar).
	\end{align}
	
	For $r>q-\sigma$ note that the residue vanishes as there is not any corresponding pole in the expression. At the same time, in the formula for $\mathcal{A}^{\dagger}$ if we look at the third line from below in \eqref{eq:Ainvcoefs} we can see that since in this case we have $i_1+\ldots+i_m<s_1+\ldots+s_m$, there exists at least one $j$ such that $i_j<s_j$ and thus the corresponding $(i_j)_{s_j}$ vanishes, together with the whole summand. Thus for $r>q-\sigma$ the expressions in the top line and in the bottom line of \eqref{eq:AcAirel} are both zero and therefore agree as well.
	
	This proves the proposition.
\end{proof}

\section{Quasi-polynomiality of the BMS numbers} \label{sec:bmsquasipol}
In the present section we prove our main result: the quasi-polynomiality of the BMS numbers. 
Note that in terms of $\mathcal{A}$-operators equation \eqref{eq:bAexpr} takes the following form:
\begin{equation}\label{eq:bacor}
b^\circ_{g,\mu} = \prod_{i=1}^n \dfrac{(m\mu_i-m)!}{\mu_i!\,(m\mu_i-\mu_i-1)!} \;\cdot\;[\hbar^{2g-2+n}] \left\langle\prod\limits_{i=1}^n\mathcal{A}(\mu_i,\hbar) \right\rangle^{\circ}
\end{equation}

\begin{theorem}\label{th:apoly}
	In the stable case, i.e. for $(g,n)\notin \{(0,1),\,(0,2)\}$ we have
	\begin{equation}\label{eq:bconncorrqpol}
	b^\circ_{g,\mu} = \prod_{i=1}^n \dfrac{(m\mu_i-m)!}{\mu_i!\,(m\mu_i-\mu_i-1)!} \;\cdot\;  \dfrac{\mathrm{Poly}_{g,n}(\mu_1,\dots,\mu_n)}{\displaystyle\prod_{i=1}^n\prod_{\substack{m\leq j_i \leq 4g-4+2n-1 \\ m\nmid j_i}} \left(\mu_i-\dfrac{j_i}{m}\right)},
	\end{equation}
	where $\mathrm{Poly}_{g,n}(\mu_1,\dots,\mu_n)$ is a polynomial in $\mu_1,\dots,\mu_n$.
\end{theorem}
\begin{proof}
From \eqref{eq:bacor} we see that in order to prove the present theorem it is sufficient to prove the following statement for $(g,n)\notin \{(0,1),\,(0,2)\}$:
\begin{equation}\label{eq:conncorrqpol}
[\hbar^{2g-2+n}] \left\langle\prod\limits_{i=1}^n\mathcal{A}(\mu_i,\hbar) \right\rangle^{\circ} = \dfrac{\mathrm{Poly}_{g,n}(\mu_1,\dots,\mu_n)}{\displaystyle\prod_{i=1}^n\prod_{\substack{m\leq j_i \leq 4g-4+2n-1 \\ m\nmid j_i}} \left(\mu_i-\dfrac{j_i}{m}\right)},
\end{equation}
for $\mathrm{Poly}_{g,n}(\mu_1,\dots,\mu_n)$ being some polynomial in $\mu_1,\dots,\mu_n$.

From Proposition \ref{prop:AcoefsPoly} and Definition \ref{def:Aoperdef} we have:
\begin{equation}\label{eq:Aopercoefsexpr}
\mathcal{A}(k,\hbar) = \sum_{q=-k}^{\infty}\; \sum_{l\in \mathbb{Z}+\frac{1}{2}}\; \sum_{p=0}^{\infty} \hbar^{p+q}\; \mathfrak{A}_{p,l,q}(k)\, E_{l-q,l} + 
\sum_{p=-1}^\infty \hbar^p\,\mathfrak{I}_p(k)\mathrm{Id},
\end{equation}
where (for $q>0$)
\begin{equation}
\mathfrak{A}_{p,l,q}(k)=\dfrac{1}{\displaystyle\prod_{\substack{m\leq j \leq 2p-1 \\ m\nmid j}} \left(k-\dfrac{j}{m}\right)}\cdot\dfrac{\mathrm{S}_{p,l,q}(k)}{(k+1)(k+2)\cdots(k+q)},
\end{equation}
where $S_{p,l,q}(k)$ are the polynomials from \eqref{eq:AcoefsPoly}, and
\begin{equation}\label{eq:Iexpr}
\mathfrak{I}_{p}(k)=\dfrac{1}{\displaystyle\prod_{\substack{m\leq j \leq 2p-1 \\ m\nmid j}} \left(k-\dfrac{j}{m}\right)}\cdot \dfrac{\mathrm{S}^{\Id}_{p}(k)}{k^2\, (mk-k+1)},
\end{equation}
where $\mathrm{S}^{\Id}_{p}(k)$ are the polynomials from Proposition \ref{prop:AIDcoefsPoly}.
%
Note that the sum over $q$ in \eqref{eq:Aopercoefsexpr} starts from $q=-k$, since for $q<-k$ the expression in e.g. the second line of \eqref{eq:AEcoef} vanishes.

Consider the disconnected correlator
\begin{equation}\label{eq:correlcoefs}
[\hbar^{2g-2+n}]\left\langle \mathcal{A}(\mu_1,\hbar)\dots\mathcal{A}(\mu_n,\hbar)\right\rangle
\end{equation}
Recall \eqref{eq:Aopercoefsexpr}.
Consider the $\Id$-part of $\mathcal{A}$. Note that the inclusion-exclusion formula for the connected correlator in terms of the disconnected ones for $n\geq 2$ will always contain the term of the form 
\begin{equation}
\langle \mathcal{A}(\mu_1,\hbar)\rangle\langle \mathcal{A}(\mu_2,\hbar)\dots\mathcal{A}(\mu_n,\hbar)\rangle,
\end{equation} 
with the opposite sign. Note that in the one-point correlator only the $\Id$-part gives a nonzero contribution. More precisely, 
\begin{equation}
\langle \mathcal{A}(\mu_1,\hbar)\rangle = \sum_{p=-1}^\infty \hbar^p\,\mathfrak{I}_p(k)
\end{equation}
This is precisely the factor which the $\Id$-part of the operator $\mathcal{A}(\mu_1,\hbar)$ contributes to $\langle \mathcal{A}(\mu_1,\hbar) \mathcal{A}(\mu_2,\hbar)\dots\mathcal{A}(\mu_n,\hbar)\rangle$, i.e.
\begin{align}
&
\left\langle \sum_{p=-1}^{\infty} \hbar^p\, \mathfrak{I}_p\, \Id\; \mathcal{A}(\mu_2,\hbar)\dots\mathcal{A}(\mu_n,\hbar)\right\rangle 
\\ \notag &
= \langle \mathcal{A}(\mu_1,\hbar)\rangle\langle \mathcal{A}(\mu_2,\hbar)\dots\mathcal{A}(\mu_n,\hbar)\rangle
\end{align}
This means that these contributions precisely cancel in the inclusion-exclusion formula. Similar reasoning proves that for $n\geq 2$ the $\Id$-parts of the $\mathcal{A}$-operators do not give any nonzero contributions into the connected correlator at all. Thus any connected multi-point correlator can only have poles
coming from the $\Id$-less parts of the $\mathcal{A}$-operator. In what follows we will assume that $n>1$ and drop the $\Id$-part altogether, and after that we will treat the $n=1$ case separately at the very end of this proof of the present proposition.

Note that
\begin{equation}\label{eq:leftvacact}
\forall q \leq 0 \quad \langle 0 | E_{l-q,l} = 0,
\end{equation}
and thus only coefficients $\mathfrak{A}_{p,l,q}(\mu_1)$ with $q>0$ will appear in the correlator \eqref{eq:correlcoefs}.
	
Note that a correlator of the form
\begin{equation}
\left\langle \left(\sum_{l_1\in\mathbb{Z}+\frac{1}{2}}c_{l_1,q_1}E_{l_1-q_1,l_1}\right)\cdots \left(\sum_{l_n\in\mathbb{Z}+\frac{1}{2}}c_{l_n,q_n}E_{l_n-q_n,l_n}\right)\right\rangle 
\end{equation}
can only be nonzero if $\sum_{i=1}^n q_i = 0$.
Thus when one writes the $\mathcal{A}$-operators in terms of the sums over $q$ \eqref{eq:Aopercoefsexpr} in the disconnected correlator \eqref{eq:correlcoefs}
and expands the brackets, only the terms with  $\sum_{i=1}^n q_i = 0$ will survive. Also note that in any surviving term in this expression we have $q_1\geq 0$, due to \eqref{eq:leftvacact}. Finally, note that from \eqref{eq:Aopercoefsexpr} we have $q_i\geq -\mu_i$. Thus, for fixed $\mu_2,\dots,\mu_n$ we have
\begin{equation}
0\leq q_1 \leq \mu_2+\dots+\mu_n
\end{equation} 
and
\begin{equation}
\forall i\in\{2,\dots,n\} \quad -\mu_i \leq q_i \leq \sum_{\substack{2\leq j \leq n\\j\neq i}}\mu_j
\end{equation}
Thus when one plugs \eqref{eq:Aopercoefsexpr} into the correlator \eqref{eq:correlcoefs}
and expands all $q$-, $l$- and $p$-sums only a finite number of terms survives. 
Indeed, the fact that $q$-sums become finite is implied by the above reasoning; $p$-sums become finite since (after we dropped the $\Id$-part) we have only non-negative powers of $u$ entering the picture and we collect a coefficient in front of a specific power of $u$; and $l$-sums become finite since we have vacuum on both sides and since we already know that the $q$-sums are finite.

Since 
$S_{p,l,q}(k)$ are 
polynomials, this implies that correlator \eqref{eq:correlcoefs}
is a rational function in $\mu_1$ for fixed $\mu_2,\dots,\mu_n$, as a finite sum of rational functions. Moreover, we see that it can have at most simple poles at the negative integer points, at most simple poles at points $\frac{j}{m}$, $m\leq j \leq 4g-4+2n-1,\; m\nmid j$, and no other poles.

The fact that the disconnected correlators are rational functions in $\mu_1$ automatically implies that the connected ones are rational in $\mu_1$ as well. In order to prove that stable connected correlators are quasi-polynomial in $\mu_1$, i.e. they have the form \eqref{eq:conncorrqpol}, we only need to prove that they do not have poles at negative integer points.

Let us prove that  for any $r\in\mathbb{Z}$ we have
\begin{equation}\label{eq:corrreszero1}
\mathop{\Res}_{\mu_1 = - r}\left\langle \mathcal{A}(\mu_1,\hbar)\mathcal{A}(\mu_2,\hbar)\cdots\mathcal{A}(\mu_n,\hbar)\right\rangle^\circ = 0.
\end{equation}
Note that, from the definition \eqref{eq:Aoperdef} of $\mathcal{A}$-operators, this is equivalent to the following statement:
\begin{equation}\label{eq:corrreszero}
\mathop{\Res}_{\mu_1 = - r}\left\langle \mathcal{A}(\mu_1,\hbar)\check{\mathcal{A}}(\mu_2,\hbar)\cdots\check{\mathcal{A}}(\mu_n,\hbar)\right\rangle^\circ = 0.
\end{equation}
For $r\notin \mathbb{Z}_{>0}$ it is clear. 
Consider $r\in \mathbb{Z}_{>0}$. 
From Proposition~\ref{prop:AinvcoefsPoly}, for the disconnected correlator we have
\begin{equation}
\mathop{\Res}_{\mu_1 = - r}\left\langle \mathcal{A}(\mu_1,\hbar)\prod_{i=2}^n \check{\mathcal{A}}(\mu_i,\hbar)\right\rangle = 
c(r)
\left\langle\check{\mathcal{A}}^{\dagger}(r,\hbar) \prod_{i=2}^n \check{\mathcal{A}}(\mu_i,\hbar)\right\rangle,
\end{equation}
where $c(r)$ is the coefficient in Proposition~\ref{prop:AinvcoefsPoly}.
%
Recalling equations \eqref{eq:aopdef} and \eqref{eq:ainvopdef} we can see that the RHS of the previous equality reduces to
\begin{equation}\label{eq:rescord}
C\left\langle\exp \left( \sum_{i=1}^\infty \frac{\alpha_{i} p^*_i}{i} \right)
D(u)^m
	\alpha_r \prod_{i=2}^n \frac{\alpha_{-\mu_i}}{\mu_i}\right\rangle
\end{equation}
for some specific coefficient \( C\) that depends only on $r$ (and $m$).
Because \( [\alpha_k, \alpha_l] = k\delta_{k+l,0} \), and \( \alpha_{r} \) annihilates the vacuum, this residue is zero unless one of the \( \mu_i \) equals \( r \) for \( i \geq 2\).

Now return to the connected $n$-point correlator for $n>2$. It can be calculated from the disconnected one by the inclusion-exclusion principle, so in particular it is a finite sum of products of disconnected correlators. Hence the connected correlator is also a rational function in $\mu_1$, and all possible poles must be inherited from the disconnected correlators.
The above reasoning implies that we can assume \( \mu_i = r \) for some \( i \geq 2\). Without loss of generality we can assume that it is the case for $i=2$, and $\mu_i\not= r $ for $i\geq 3$.  Then we get a contribution from \eqref{eq:rescord}, but this is canceled in the inclusion-exclusion formula exactly by the term coming from
\begin{align}
&
\mathop{\Res}_{\mu_1 = - r}\left\langle\mathcal{A}(\mu_1,\hbar)\check{\mathcal{A}}(r,\hbar) \right\rangle\left\langle 
	\prod_{i=3}^n
	\check{\mathcal{A}}(\hbar, \mu_j)\right\rangle
\\ \nonumber 
&= C\left\langle e^{\alpha_1}
D(\hbar)^m
	\alpha_r\, \alpha_{-r}\right\rangle\left\langle e^{\alpha_1}
D(\hbar)^m
	\prod_{i=3}^n \alpha_{-\mu_i}\right\rangle \\ \nonumber
&= C \left\langle e^{\alpha_1}
	D(\hbar)^m
	\alpha_r \prod_{i=2}^n \alpha_{-\mu_i}\right\rangle
\end{align}
Thus we have proved \eqref{eq:corrreszero1} for $n>2$, which implies that for $n>2$ we have 
\begin{equation}
[\hbar^{2g-2+n}] \left\langle\prod\limits_{i=1}^n\mathcal{A}(\mu_i,\hbar) \right\rangle^{\circ} = \dfrac{\widetilde{\mathrm{Poly}}_{g,n}(\mu_1,\dots,\mu_n)}{\displaystyle\prod_{\substack{m\leq j \leq 2p-1 \\ m\nmid j}} \left(\mu_1-\dfrac{j}{m}\right)},
\end{equation}
where $\widetilde{\mathrm{Poly}}_{g,n}(\mu_1;\mu_2,\dots,\mu_n)$ is some expression polynomial in $\mu_1$, and $p=2g-2+n$.

Note that since polynomials $\mathrm{S}_{p,l,q}$ had their degree bounded from above by $6p+q$, the degree of $\widetilde{\mathrm{Poly}}_{g,n}$ as a polynomial in $\mu_1$ is bounded from above by $6p$, since, as follows from what we have shown above, the factors $(\mu_1+1)\cdots(\mu_1+q)$ in the denominator get canceled in the connected correlators, and thus the degree of the polynomial in the numerator gets decreased by $q$, thus resulting in the $6p$ bound.

Since the degree of $\widetilde{\mathrm{Poly}}_{g,n}(\mu_1;\mu_2,\dots,\mu_n)$ as a polynomial in $\mu_1$ is bounded from above by a number independent of $\mu_2,\dots,\mu_n$, and since the correlator \eqref{eq:correlcoefs} is symmetric in $\mu_1,\dots,\mu_n$, this implies the statement of the theorem, apart from the 1- and 2-point correlator cases.

Consider the 2-point correlators. Let us prove that they do not have poles at negative integers for the nonzero genus case. For $r\in\mathbb{Z}_{>0}$ we have
\begin{align}
&\mathop{\Res}_{\mu_1 = - r}[\hbar^p]\cord{\mathcal{A}(\mu_1,\hbar)\mathcal{A}(\mu_2,\hbar) }
= C[\hbar^p]\cord{e^{\alpha_1}
	D(\hbar)^m
	\alpha_r\, \alpha_{-\mu_2}}.
\end{align}
This is nonzero only for $\mu_2=r$ and equal to
\begin{equation}
C[\hbar^p]\cord{e^{\alpha_1}
	D(\hbar)^m}=
C[\hbar^p]\; 1
\end{equation}
which is obviously nonzero only for $p=0$, i.e. for the unstable case. Again, the fact that correlator \eqref{eq:correlcoefs} for $n=2$ is still symmetric in $\mu_1$ and $\mu_2$ and the analogous reasoning regarding the upper bound on the degree of the polynomial in $\mu_1$ in the numerator as in the case of $n>2$ imply that for $n=2$ the theorem holds as well.

What remains is to prove that stable one-point correlators are polynomial. Note that
\begin{equation}
[\hbar^p]\left\langle\mathcal{A}(k,\hbar)\right\rangle=\mathfrak{I}_p(k),
\end{equation}
since for any $i,\,j$ we have $\left\langle E_{i,j}\right\rangle=0$. Propositions \ref{prop:AIDcoefsPoly} and \ref{prop:IdDivis} together imply that for $p\geq 0$ the coefficient $\mathfrak{I}_p(k)$ (given in \eqref{eq:Iexpr}) is in fact a polynomial in $k$, which is precisely what we need.

Thus we conclude that the statement in \eqref{eq:conncorrqpol} holds for $n=1$ and $n=2$ as well, as long as $(n,k)\notin \{(1,-1),(2,0)\}$.	Together with what is written above this implies the statement of the theorem.

\end{proof}

The following corollary substantiates, as discussed in the Introduction and Section~\ref{sec:toporec-proof}, our motivation in proving the specific combinatorial structure of the BMS numbers provided by the theorem above. Recall the definition of the space $\Xi^d$ in Section~\ref{sec:specfuncs}. 
\begin{corollary} \label{corollaryW}
	For any $g\geq 0$ and $n\geq 1$, $2g-2+n >0$, there exists $d\geq 0$ and a function $W_{g,n}(z_1,\dots,z_n)\in \left(\Xi^d\right)^{\otimes n}$ such that
	\begin{equation}
	[X^{\mu_1}_1\cdots X^{\mu_n}_n]W_{g,n} = b^\circ_{m,g;\mu_1,\dots,\mu_n}.
	\end{equation}	
	Here $[X^{\mu_1}_1\cdots X^{\mu_n}_n]W_{g,n}$ corresponds to taking the coefficient in front of $X^{\mu_1}_1\cdots X^{\mu_n}_n$ in the series expansion of $W_{g,n}$ at $X_i\rightarrow 0$. 
\end{corollary}
\begin{proof}
	Follows from Theorem \ref{th:apoly} and Proposition \ref{prop:xiexp}.
\end{proof}

In particular, this corollary proves the structural statement of the $n$-point generating functions of the BMS numbers that we mentioned in Remark~\ref{rem:remark2xi}.

\appendix
\section{Around Faulhaber's formula} \label{sec:appfaul}
In the present paper we need several technical results related to the classical Faulhaber formula.

\begin{proposition}\label{prop:extFaulhaber}
	Expression
	\begin{equation}\label{eq:extFaulhaber}
	T_d(x,k):=\sum_{0\leq\gamma_1<\ldots<\gamma_d< k}(x+\gamma_1)\cdots(x+\gamma_d)
	\end{equation}
	is a polynomial in $x$ and $k$ of total degree $2d$.
\end{proposition}
\begin{proof}
	Recall the classical Faulhaber formula:
	\begin{equation}\label{eq:Faulhaber}
	\sum_{k=1}^n k^{p} = \frac{n^{p+1}}{p+1}+\frac{1}{2}n^p+\sum_{k=2}^p \;(p)_{k-1}\;n^{p-k+1}\;\frac{B_{k}}{k!},
	\end{equation}
	where $B_k$ are the Bernoulli numbers.
	Note that this formula implies that $\sum_{\gamma=0}^{k-1}\gamma^p$ is a polynomial in $k$ of degree $p+1$. Applying this repeatedly to the sum \eqref{eq:extFaulhaber} to successively eliminate the sum in $\gamma_1$, then the sum in $\gamma_2$, and so on, one can easily see that the statement of the proposition holds. 
\end{proof}

\begin{proposition}\label{prop:TtTrel}
	For the polynomial $T_d(x,k)$ defined in \eqref{eq:extFaulhaber} and 
	\begin{equation}
	\wt{T_d}(x,k):=\sum_{1\leq\gamma_1\leq\ldots\leq\gamma_d\leq k}(-x+\gamma_1)\cdots(-x+\gamma_d)
	\end{equation}	
	we have
	\begin{equation}\label{eq:TTtrel}
	\wt{T_d}(x,k)=T_d(x,-k)
	\end{equation}
	By taking $T_d(x,-k)$ we mean first computing the sum in \eqref{eq:extFaulhaber} (assuming that $d\geq k$ is a positive integer) and obtaining a polynomial in $k$ as shown in Proposition \ref{prop:extFaulhaber}, and only after that substituting $-k$ in place of $k$.
\end{proposition}
\begin{proof}
	First note that, analogous to Proposition \ref{prop:extFaulhaber}, $\wt{T_d}(x,k)$ is a polynomial in $k$ and $x$.
	
	Let us prove the current proposition by induction in $d$. For $d=1$ the proof is straightforward:
	\begin{align}
	T_1(x,k) &= \sum_{0\leq\gamma_1 < k}(x+\gamma_1) = \dfrac{1}{2}\, k\, (k+2x-1)\\ \nonumber
	\widetilde{T}_1(x,k) &= \sum_{1\leq\gamma_1 \leq k}(-x+\gamma_1) = \dfrac{1}{2}\, k\, (k-2x+1)
	\end{align}
	Let us assume that $\wt{T_i}(x,k)=T_i(x,-k)$ for all $i<d$. The definitions of $T$ and $\widetilde{T}$ directly imply that
	\begin{align}\label{eq:Tbbmrel}
	T_d(x,k)-T_d(x,k-1) &= (x+k-1)\; T_{d-1}(x,k-1) \\ \label{eq:Ttbbmrel}
	\widetilde{T}_d(x,k)-\widetilde{T}_d(x,k-1) &= \sum_{i=1}^d (-x+k)^i \;\widetilde{T}_{d-i}(x,k-1)
	\end{align}
	With the help of these two identities and the induction hypothesis we have 
	\begin{align}\label{eq:TtTcomp}
	&T_p(x,-b+1)-T_p(x,-b) \mathop{=}^{\eqref{eq:Tbbmrel}} (x-b)\; T_{p-1}(x,-b) \mathop{=}^{\eqref{eq:TTtrel}} (x-b)\; \wt{T}_{p-1}(x,b) \\ \nonumber
	&\mathop{=}^{\eqref{eq:Ttbbmrel}}(x-b) \left(\widetilde{T}_{p-1}(x,b-1) + \sum_{i=1}^{p-1} (-x+b)^i \;\widetilde{T}_{p-1-i}(x,b-1)\right)\\ \nonumber
	&=-(-x+b)\; \widetilde{T}_{p-1}(x,b-1) - \sum_{i=1}^{p-1} (-x+b)^{i+1} \;\widetilde{T}_{p-1-i}(x,b-1) \\ \nonumber
	&=-\sum_{i=1}^{p} (-x+b)^{i+1} \;\widetilde{T}_{p-i}(x,b-1) \\ \nonumber
	&\mathop{=}^{\eqref{eq:Ttbbmrel}}-\left(\widetilde{T}_p(x,b)-\widetilde{T}_p(x,b-1)\right)\\ \nonumber
	&=\widetilde{T}_p(x,b-1)-\widetilde{T}_p(x,b)
	\end{align}
	Now let us prove that $\wt{T_p}(x,b)=T_p(x,-b)$ holds for all nonnegative integer values of $b$ by induction in $b$. For $b=0$ both of these expressions are equal to $0$ since both of the polynomials $\wt{T_p}(x,b)$ and $T_p(x,b)$ are evidently divisible by $b$. Suppose that $\wt{T_p}(x,b-1)=T_p(x,-(b-1))$. Then, since the LHS of the top line of \eqref{eq:TtTcomp} is equal to the bottom line, we have
	\begin{equation}
	T_p(x,-b)-\wt{T_p}(x,b) \mathop{=}^{\eqref{eq:TtTcomp}} T_p(x,-(b-1))-\wt{T_p}(x,b-1)=0
	\end{equation}
	This proves that $\wt{T_p}(x,b)=T_p(x,-b)$ for all nonnegative integer values of $b$; thus they also coincide as polynomials in $b$, which proves the proposition.
\end{proof}
Let us note\footnote{We would like to thank Yu.~Burman for this comment.} that if we consider the polynomial $T_d(x,k)$ as an Ehrhart polynomial of some polytope then the previous proposition is just the Ehrhart--Macdonald reciprocity \cite{Mac}.

\begin{proposition}\label{prop:extFaulhaberDivis}
	Polynomial	$T_d(x,k)$ defined in \eqref{eq:extFaulhaber} is divisible by $k$ for $d\geq 1$.
\end{proposition}
\begin{proof}
	Note that for $d\geq 1$ the polynomial $\wt{T_d}(x,k)$ defined in proposition \ref{prop:TtTrel} is divisible by $k$. This immediately follows from Faulhaber's formula \eqref{eq:Faulhaber}, as after taking the sum over $\gamma_1,\dots,\gamma_{d-1}$ we will have a sum over $\gamma_d$ from $1$ to $k$ of a polynomial expression in $\gamma_d$ and from \eqref{eq:Faulhaber} we clearly see that for each individual power the sum is divisible by $k$. Then Proposition \ref{prop:TtTrel} implies the present proposition, as $\wt{T_d}(x,k)=T_d(x,-k)$ and $T_d(x,-k)$ is divisible by $k$ if and only if $T_d(x,k)$ is divisible by $k$.
\end{proof}

\end{document}